\documentclass[12pt]{amsart}
\usepackage[margin=1in]{geometry}

\usepackage{url,amsmath,amssymb,mathrsfs}
\usepackage{enumitem}
\usepackage{commath}
\usepackage{physics}
\usepackage{nicefrac}
\usepackage{amsthm}
\usepackage{xfrac}
\usepackage[all]{xy}
\usepackage{graphicx}
\graphicspath{ {./images/} }

\newtheorem{theorem}{Theorem}[section]
\newtheorem{lemma}[theorem]{Lemma}
\newtheorem{proposition}[theorem]{Proposition}
\newtheorem{corollary}[theorem]{Corollary}

\newtheorem{innercustomgeneric}{\customgenericname}
\providecommand{\customgenericname}{}
\newcommand{\newcustomtheorem}[2]{%
	\newenvironment{#1}[1]
	{%
		\renewcommand\customgenericname{#2}%
		\renewcommand\theinnercustomgeneric{##1}%
		\innercustomgeneric
	}
	{\endinnercustomgeneric}
}
\newcustomtheorem{customthm}{Theorem}
\newcustomtheorem{customlemma}{Lemma}

\theoremstyle{definition}
\newtheorem{example}[theorem]{Example}
\newtheorem{definition}[theorem]{Definition}

\newcommand{\kron}[2]{\left(\frac{#1}{#2}\right)}
\newcommand{\bZ}{\mathbb Z}
\newcommand{\bQ}{\mathbb Q}
\newcommand{\bN}{\mathbb N}

\newcommand{\bC}{\mathbb C}
\newcommand{\Rbar}{\overline{R}}

\newcommand{\Inv}{\textup{Inv}}

\newcommand{\Cl}{\textup{Cl}}

\newcommand{\modulo}[1]{\textup{ (mod }#1)}

\newcommand{\st}{\textup{ s.t. }}

\newcommand{\ow}{\textup{otherwise}}

\newcommand{\spec}{\textup{Spec}}

\newcommand{\cO}{\mathcal{O}}
\newcommand{\Irr}{\textup{Irr}}
\newcommand{\term}[1]{\textbf{\textup{#1}}}

\newcommand{\quot}[2]{\large\sfrac{#1}{#2}\normalsize}

\title[Multiplicative Subring Relations]{Multiplicative Relationships of Subrings and their Applications to Factorization}
\author{Grant Moles}
\date{November 2024}

\begin{document}
	
	\begin{abstract}
		When studying the properties of a ring $R$, it is often useful to compare $R$ to other rings whose properties are already known. In this paper, we define three ways in which a subring $R$ might be compared to a larger ring $T$: being associated, being ideal-preserving, or being locally associated. We then explore how these properties of a subring might be leveraged to give information about $R$, including applications to the field of factorization. Of particular interest is the result that an order in a number field is associated if and only if it is both ideal-preserving and locally associated. We conclude with a discussion of how these properties are realized in the case of orders in a number field and how such orders might be found.
	\end{abstract}
	
	\maketitle
	
	\section{Introduction}
	
	Suppose that we have a ring $T$ with subring $R$. If we know a great deal of information about the structure of $R$, we might be able to determine information about $T$ by considering relationships that might exist between $R$ and $T$. For instance, if we happen to know that $T$ is a polynomial ring (\cite{coykendallpolynomial}, \cite{CO1}, \cite{Co2005}), a ring of formal power series (\cite{MO1}, \cite{dissertation}, \cite{radicalconductor}, \cite{Co2005}), or an overring of $R$ (\cite{overrings}), we can often extrapolate useful information about factorization in $T$ from known properties of $R$. This paper will explore the inverse: if we know information about $T$, what relationships between $T$ and its subring $R$ might we use to extrapolate information about $R$? In particular, we will look at relationships among the multiplicative structures of these rings, primarily as they relate to factorization.
	
	The following theorem from Franz Halter-Koch \cite{halter-koch} provides a motivating example. We will state this theorem using a different notation than in the original paper in order to match more closely what we will see later in this paper.
	\begin{theorem}
		\label{halter}
		Let $R$ be the order of index $n$ in the quadratic number field $K=\bQ[\sqrt{d}]$ for some $n\in\bN\backslash\{1\}$. That is, if $d\equiv 1\modulo{4}$, then $R=\bZ[n\frac{1+\sqrt{d}}{2}]$; otherwise, $R=\bZ[n\sqrt{d}]$. Then $R$ is a half-factorial domain if and only if the following conditions all hold:
		\begin{enumerate}
			\item $\Rbar=\cO_K$ is an HFD;
			\item $\Rbar=R\cdot U(\Rbar)$;
			\item $n$ is either a prime number or twice an odd prime.
		\end{enumerate}
	\end{theorem}
	
	In this case, we note that we are relating the elasticity of an order in a number field $R$ to that of its integral closure $\Rbar$, the full ring of algebraic integers in $K$. This is particularly useful because of a result from Narkiewicz \cite{narkiewicz} which builds upon a classical result from Carlitz \cite{carlitz}, as follows:
	\begin{theorem}
		\label{davenport elasticity}
		Let $K$ be an algebraic number field, $R=\cO_K$, the ring of algebraic integers in $K$, and $G=\Cl(R)$, the ideal class group of $R$. If $\abs{G}=1$, then $R$ is a UFD and thus $\rho(R)=1$. Otherwise,
		$$\rho(R)=\frac{D(G)}{2},$$ where $D(G)$ is the Davenport constant of the finite abelian group $G$. In particular, $R$ is an HFD (i.e. $\rho(R)=1$) if and only if $\abs{G}\leq 2$.
	\end{theorem}
	
	This theorem allows us to determine the elasticity of a number ring from its ideal class group. What Theorem \ref{halter} really allows us to do is take what we know about rings of algebraic integers, in particular whether they are half-factorial, and leverage that into determining whether a more generic order in a number field is half-factorial. This idea of using factorization properties of a number ring to study the orders contained within has proven to be a particularly rich vein of research, as in \cite{choi2024class}, \cite{Kettinger2024}, and \cite{primeconductor}.
	
	In more abstract terms, these results use a ring $T$ whose multiplicative structure is more well-understood, (in this case, a ring of algebraic integers) properties of that ring (in Theorem \ref{halter}, that it is an HFD), and relationships between the ring and a subring $R$ (in Theorem \ref{halter}, that $\Rbar=R\cdot U(R)$ and that the index is prime or twice an odd prime) to determine a property of $R$ (in Theorem \ref{halter}, that it is also an HFD).
	
	The following theorems from \cite{rago} and \cite{radicalconductor}, respectively, give generalizations of Halter-Koch's theorem which further demonstrate the utility of the relationship $\Rbar=R\cdot U(\Rbar)$ to factorization within orders in a number field.
	\begin{theorem}
		\label{rago hfd}
		Let $R$ be an order in a number field $K$ with conductor ideal $I$. Let $I=P_1^{a_1}\dots P_k^{a_k}$ be the factorization of $I$ into $\Rbar$-ideals, and denote $Q_i:=R\cap P_i$ for each $1\leq i\leq k$. Then $R$ is an HFD if and only if the following conditions hold:
		\begin{enumerate}
			\item $\Rbar$ is an HFD;
			\item $\Rbar=R\cdot U(\Rbar)$;
			\item For each $1\leq i\leq k$, $a_i\leq 4$, and letting $\pi_i$ be an arbitrary prime element of $\Rbar_{Q_{i}}$, $v_{\pi_i}(\Irr(R_{Q_i}))\subseteq\{1,2\}$. Moreover, if $P_i$ is a principal ideal, then $a_i\leq 2$ and $v_{\pi_i}(\Irr(R_{Q_i}))=\{1\}$.
		\end{enumerate}
	\end{theorem}
	
	\begin{theorem}
		\label{elasticities equal}
		Let $R$ be an order in a number field $K$ with conductor ideal $I$. If $\Rbar=R\cdot U(\Rbar)$ and $I$ is a radical $\Rbar$-ideal, then $\rho(R)=\rho(\Rbar)$ and $\rho(R[[x]])=\rho(\Rbar[[x]])$.
	\end{theorem}
	
	Then when considering multiplicative relationships between a ring and its subring, something similar to $\Rbar=R\cdot U(\Rbar)$ might be a good candidate for a relationship of interest. In fact, these examples (and this relationship in particular) motivated this direction of inquiry.
	
	Throughout the remainder of this paper, we will define particular types of subrings, each of which will have a relationship to the ring in which it is contained similar to the ones explored here. We will then see what properties such subrings have, how these properties to apply to the case of orders in a number field, and how we might find orders with these subring relationships.
	
	\section{Multiplicative Subring Relationships}
	We will now define a few particular types of subrings and the properties they possess. The definitions presented here originally appeared in \cite{dissertation} and \cite{radicalconductor}, though we will explore these properties more in-depth. The first comes from applying the relationship $\Rbar=R\cdot U(\Rbar)$ which we encountered above to a more general ring extension.
	\begin{definition}
		Let $T$ be a commutative ring with identity and $R\subseteq T$ a subring (not necessarily with identity). We say that $R$ is an \term{associated subring} of $T$ if $T=R\cdot U(T)$; that is, if for any $t\in T$, there exist $r\in R$ and $u\in U(T)$ such that $t=ru$.
	\end{definition}
	
	The following examples give an idea of what associated subrings might look like; first, we have a few trivial examples.
	
	\begin{example}
		Any commutative ring with identity $T$ is an associated subring of itself. Let $T$ be any field and $R$ any nonzero subring of $T$. Then $R$ is an associated subring of $T$. Let $R$ be an associated subring of a ring $T$, and let $S$ be any subring of $T$ containing $R$ (i.e. $R\subseteq S\subseteq T$). Then $S$ is an associated subring of $T$.
	\end{example}
	
	\begin{example}
		Let $R=\bZ[\sqrt{5}]$ and $T=\Rbar=\bZ[\frac{1+\sqrt{5}}{2}]$, and let $u=\frac{1+\sqrt{5}}{2}\in U(T)$. Then any element of $T$ is of the form $t=a+bu$. If $b$ is even, then $t\in R$; if $a$ is even, then $t=ru$ for some $r\in R$; if $a$ and $b$ are of the same parity, then $t=ru^2$ for some $r\in R$. Then $R$ is an associated subring of $T$.
	\end{example}
	
	\begin{example}
		Let $m,n\in\bN$ with $m|n^i$ for some $i\in\bN$, and let $T=\bZ[\frac{1}{n}]$ and $R=m\bZ\subseteq T$. Note that any element of $T$ is of the form $t=\frac{a}{n^k}$ for some $a\in\bZ$ and $k\in\bN_0$. Then $t=(n^ia)\cdot \frac{1}{n^{i+k}}$, with $n^ia\in R$ and $\frac{1}{n^{i+k}}\in U(T)$. Then $R$ is an associated subring of $T$. Note that in this case, for $m\neq 1$, $R$ is a subring without identity.
	\end{example}
	
	As one might expect, the term ``associated" here is used to reference the fact that every element of $T$ has an associate which lies in the subring $R$. From this definition, we might also make the observation that there are two trivially equivalent ways of thinking about this relationship. First, as presented here, any $t\in T$ has some $r\in R$ and $u\in U(T)$ such that $t=ru$. On the other hand, for every $t\in T$, there is some $u\in U(T)$ such that $tu\in R$. Though this equivalence is immediately obvious, we note it here because it may often be convenient to think of the relationship in the second form, rather than the first.
	
	We will of course explore this property more in a moment; however, we will first introduce two more types of subrings. Although it may not be immediately obvious that these types of subrings are at all related to one another, we will see later that they are actually very naturally connected.
	
	\begin{definition}
		Let $T$ be a commutative ring and $R$ a subring (neither assumed to have identity). We say that $R$ is an \term{ideal-preserving subring} of $T$ if, for any $T$-ideals $J_1\nsubseteq J_2$, we have $R\cap J_1\nsubseteq J_2$ (equivalently, $R\cap J_1\nsubseteq R\cap J_2$).
	\end{definition}
	
	\begin{example}
		Any commutative ring $T$ is an ideal-preserving subring of itself. Let $T$ be any field and $R$ any nonzero subring of $T$. Then $R$ is an ideal-preserving subring of $T$. Let $R$ be an ideal-preserving subring of a ring $T$, and let $S$ be any subring of $T$ containing $R$. Then $S$ is an ideal-preserving subring of $T$.
	\end{example}
	
	Just as before, the term ``ideal-preserving" is obviously appropriate for this type of subring. What the definition actually says is that containment relationships between ideals of $T$ are preserved when restricting down to ideals in $R$.
	
	\begin{definition}
		Let $T$ be a commutative ring with identity, $R\subseteq T$ a subring (with identity) of $T$, and $I:=(R:T)$, the conductor ideal from $T$ into $R$. We say that $R$ is a \term{locally associated subring} of $T$ if
		$$\quot{U(T)}{U(R)}\cong \quot{U(\sfrac{T}{I})}{U(\sfrac{R}{I})}.$$
	\end{definition}
	
	\begin{example}
		Any commutative ring with identity $T$ is a locally associated subring of itself. Let $T$ be a commutative ring with identity and $R$ a subring with identity such that $(R:T)=\{0\}$. Then $R$ is a locally associated subring of $T$. Let $T$ be any field and $R$ any subring with identity of $T$. Then $R$ is a locally associated subring of $T$ (since the conductor ideal is $I=T$ if $T=R$ and $I=\{0\}$ otherwise). 
	\end{example}
	
	Unlike the previous two definitions, the term ``locally associated" may seem a bit strange. However, as we will see in a moment, this is still an appropriate choice of name for such a subring in most cases, as it will act very much like an associated subring on a particular subset of elements. One may also notice that we did not say, as we did for associated and ideal-preserving subrings, that this relationship is inherited by any intermediate subring $S$ with  $R\subseteq S\subseteq T$. This was quite intentional; we will see after a bit more discussion that this inheritance does not occur in general.
	
	The difficulty with these definitions of associated, ideal-preserving, and locally associated subrings is that they can often be unwieldy to prove or use. After all, we may have infinitely many elements, ideals, or units in $T$ that we need to check, and there may not in general be a way to systematically check the conditions for every element or ideal as there was in the examples above. Thus, it may help in some specific cases to have an equivalent characterization that either gives more information about the subring or is easier to check. We will explore some of these characterizations as well as the relationships among these properties, beginning with the following for locally associated subrings.
	
	\begin{proposition}
		\label{la first conditions}
		Let $T$ be a commutative ring with identity, $R\subseteq T$ a subring (with identity) of $T$, and $I:=(R:T)$. The following are equivalent:
		\begin{enumerate}
			\item $R$ is a locally associated subring of $T$.
			\item Every coset in $\quot{U(\sfrac{T}{I})}{U(\sfrac{R}{I})}$ contains a unit in $T$; that is, for any $t+I\in U(\sfrac{T}{I})$, there exists some $r+I\in U(\sfrac{R}{I})$ such that $tr+\beta\in U(T)$.
			\item If $t\in T$ is relatively prime to $I$, i.e. $tT+I=T$, then there exists $r\in R$ relatively prime to $I$, i.e. $rR+I=R$, and $u\in U(T)$ such that $t=ru$.
		\end{enumerate}
	\end{proposition}
	
	\begin{proof}
		First, we will consider in general the map $\phi:U(T)\to \quot{U(\sfrac{T}{I})}{U(\sfrac{R}{I})}$ defined by $\phi(u)=(u+I)U(\quot{R}{I})$. It is trivial to see that $\phi$ is a well-defined multiplicative group homomorphism. Suppose that $u\in U(R)$; then note that $u+I\in U(\sfrac{R}{I})$, so $\phi(u)=(1+I)U(\sfrac{R}{I})$. Then $U(R)\subseteq \ker(\phi)$. On the other hand, suppose that $u\in \ker(\phi)$, i.e. $\phi(u)=(1+I)U(\sfrac{R}{I})$. Then there is some $r+I\in U(\sfrac{R}{I})$ such that $u+I=r+I$, so $u\in R$. Furthermore, note that since $u\in U(T)$, then $u^{-1}$ exists in $T$, and thus $\phi(u^{-1})=\phi(u^{-1})\phi(u)=\phi(uu^{-1})=\phi(1)=(1+I)U(\sfrac{R}{I})$. Then as before, $u^{-1}\in R$, so $u\in U(R)$. Then $\ker(\phi)=U(R)$, so by the first isomorphism theorem, $\quot{U(T)}{U(R)}\cong \phi(U(T))$.
		
		Since $\phi(U(T))$ is a subgroup of $\quot{U(\sfrac{T}{I})}{U(\sfrac{R}{I})}$, the above discussion tells us that $R$ is a locally associated subring of $T$ if and only if $\phi$ is onto. Furthermore, Condition 2 above is actually equivalent to $\phi$ being onto; to illustrate this, suppose that $\phi$ is onto. Then for every $t+I\in U(\sfrac{R}{I})$, there is some $u\in U(T)$ such that $\phi(u)=(u+I)U(\sfrac{R}{I})=(t+I)U(\sfrac{R}{I})$. then there exists $r+I\in U(\sfrac{R}{I})$ such that $u+I=(t+I)(r+I)=tr+I$, and thus there is some $\beta\in I$ such that $tr+\beta=u\in U(T)$. The converse follows by reversing the same argument. Then $1\iff 2$.
		
		Now to show that $2\iff 3$, note that $t\in T$ is relatively prime to $I$, i.e. $tT+I=T$, if and only if $t+I\in U(\sfrac{T}{I})$ (similarly, $rR+I=R$ if and only if $r+I\in U(\sfrac{R}{I})$). Then assuming Condition 2, let $t\in T$ be relatively prime to $I$. Condition 2 tells us that there is some $r+I\in U(\sfrac{R}{I})$ and $\beta\in I$ such that $tr+\beta=u\in U(T)$. Letting $s\in R$ such that $s+I=(r+I)^{-1}\in U(\sfrac{R}{I})$, this gives us that $t=us+\beta'$ for some $\beta'\in I$, and thus $t=(s+\beta'u^{-1})u$. Then $2\implies 3$.
			
		Finally, assume that Condition 3 holds and let $t+I\in U(\sfrac{T}{I})$. Since $t\in T$ is relatively prime to $I$, Condition 3 tells us that there exist $s\in R$ relatively prime to $I$ and $u\in U(T)$ such that $t=su$. Then letting $r+I=(s+I)^{-1}\in U(\sfrac{R}{I})$, we have $tr+I=u+I$, i.e. there is some $\beta\in I$ such that $tr+\beta=u\in U(T)$. Then $3\implies 2$.
	\end{proof}
	
	This characterization may help us see the relationship between associated and locally associated subrings. From Condition 3 above, we can see that to show $R$ is a locally associated subring of $T$, we simply need to show the condition for associated subrings restricted to the elements $t$ relatively prime to the conductor ideal $I$, with the added condition that the associate of $t$ which lies in $R$ must remain relatively prime to $I$. Interestingly enough, this leaves just enough wiggle room for a subring to be associated but not locally associated. Before looking at more equivalent characterizations of these subrings, we will now take a closer look at the relationships among them.
	
	\begin{theorem}
		\label{ass implies ip}
		Let $T$ be a commutative ring with identity and $R$ a subring (not necessarily with identity). If $R$ is an associated subring of $T$, then $R$ is an ideal-preserving subring of $T$. The converse does not hold in general.
	\end{theorem}
	
	\begin{proof}
		Assume that $R$ is an associated subring of $T$, and let $J_1\nsubseteq J_2$ be ideals of $T$. This means that there is some $t\in J_1\backslash J_2$. Since $R$ is an associated subring of $T$, there exists $u\in U(T)$ such that $tu=r\in R$. Now since $J_1$ is an ideal and $t\in J_1$, $r=tu\in R\cap J_1$. However, if $r=tu\in J_2$, then $ru^{-1}=tuu^{-1}=t\in J_2$, a contradiction. Then $r\in R\cap J_1\backslash J_2$, meaning that $R\cap J_1\nsubseteq J_2$. Thus, $R$ is an ideal-preserving subring of $T$.
		
		To see that the converse does not necessarily hold, let $T=\bZ[\sqrt{2}]$ and $R=\bZ[5\sqrt{2}]\subseteq T$. Showing that $R$ is an ideal-preserving subring of $T$ which is not an associated subring of $T$ will be easier using results we will develop later. To see this part of the proof elaborated upon, see Example \ref{ip not la}.
	\end{proof}
	
	\begin{theorem}
		\label{ass implies la}
		Let $T$ be a commutative ring with identity, $R\subseteq T$ a subring (with identity) of $T$, and $I:=(R:T)$. If $R$ is an associated subring of $T$ and $U(\sfrac{R}{I})=\quot{R}{I}\cap U(\sfrac{T}{I})$ (i.e. any element of $\quot{R}{I}$ which is invertible in $\quot{T}{I}$ has its inverse lying in $\quot{R}{I}$), then $R$ is a locally associated subring of $T$. In particular, this will hold when $T$ is an integral extension of $R$.
		
		In general, if $R$ is an associated subring of $T$, $R$ may fail to be a locally associated subring of $T$. Moreover, if $R$ is a locally associated subring of $T$, $R$ may fail to be an associated subring of $T$. Furthermore, it is possible for $T$ to have a locally associated subring $R$ and a non-locally associated subring $S$ such that $R\subseteq S\subseteq T$.
	\end{theorem}
	
	\begin{proof}
		Assume that $R$ is an associated subring of $T$ and $U(\sfrac{R}{I})=\quot{R}{I}\cap U(\sfrac{T}{I})$. Then for any $t\in T$ relatively prime to $I$, there exist $r\in R$ and $u\in U(T)$ such that $t=ru$ (in fact, this holds for any $t\in T$). By the third characterization of locally associated subrings presented above, we now need only show that $r$ is relatively prime to $I$ as an element of $R$, i.e. $rR+I=R$. Note that $rT+I=t(u^{-1}T)+I=tT+I=T$, so $r$ is relatively prime to $I$ as an element of $T$. In other words, $r+I\in U(\sfrac{T}{I})$. Then since $r+I\in \quot{R}{I}\cap U(\sfrac{T}{I})=U(\sfrac{R}{I})$, $r$ is also relatively prime to $I$ as an element of $R$. Then $R$ is a locally associated subring of $T$.
		
		Note that when $T$ is an integral extension of $R$, i.e. every element of $T$ is a root of a monic polynomial in $R[x]$, then $\quot{T}{I}$ is an integral extension of $\quot{R}{I}$ as an immediate consequence. Then $U(\sfrac{R}{I})=\quot{R}{I}\cap U(\sfrac{T}{I})$.
		
		To see that the condition $U(\sfrac{R}{I})=\quot{R}{I}\cap U(\sfrac{T}{I})$ is actually important to this conclusion, consider the rings $T=\bQ[x]$ and $R=\bZ+x\bZ+x^2\bQ[x]\subseteq T$. In other words, $R$ is the subring of $T$ consisting of polynomials whose constant and linear coefficients are restricted to the integers. In this case, the conductor ideal $I=(R:T)$ is $I=x^2T$ and $U(T)=\bQ$. The first thing we might note is that for any $f(x)=\frac{r_0}{s_0}+\frac{r_1}{s_1}x+\dots+\frac{r_n}{s_n}x^n\in T$, we have that $f(x)=\left(r_0s_1+r_1s_0x+\dots+\frac{r_ns_0s_1}{s_n}\right)\cdot \frac{1}{s_0s_1}\in R\cdot U(T)$. Then $R$ is an associated subring of $T$. We will now show that $R$ is not a locally associated subring of $T$.
		
		To show that $R$ is not a locally associated subring of $T$, we need to show that there exists an element of $T$ which is relatively prime to $I$ which cannot be written as a product of an element of $R$ relatively prime to $I$ times a unit in $T$. To do so, let $t=2+3x\in T$. Note that $(2+3x)(\frac{1}{2}-\frac{3}{4}x)\equiv 1\modulo{x^2}$, so $2+3x$ is relatively prime to $I=x^2 T$. Furthermore, we might note that any element of $R$ which is relatively prime to $I$ (i.e. for which $r+I\in U(\sfrac{R}{I})$) must have its constant coefficient be $\pm 1$. Then if there is some $u\in U(T)=\bQ$ such that $ut+I=2u+3ux+I\in U(\sfrac{R}{I})$, it must be the case that $u=\pm \frac{1}{2}$. However, this results in a linear coefficient of $3u=\pm\frac{3}{2}\notin R$, so in fact there is no such unit in $T$ which can be multiplied by $t$ to produce an element of $R$ which is still relatively prime to $I$. Then $R$ is an associated subring of $T$ which is not a locally associated subring of $T$.
		
		We now want to show that a locally associated subring $R$ of $T$ need not be an associated subring of $T$. To see this, let $T=\quot{\bZ[\sqrt{2}]}{(2)}$ and $R=\bZ_2\subseteq T$. In this case, note that $I=(R:T)=\{0\}$. Then trivially, $R$ is a locally associated subring of $T$. However, $U(T)=\{1+(2),1+\sqrt{2}+(2)\}$ and thus $R\cdot U(T)=\{0+(2),1+(2),1+\sqrt{2}+(2)\}\subsetneq T$. Then $R$ is not an associated subring of $T$.
		
		Finally, we want to see that there exists some tower of rings $R\subseteq S\subseteq T$ such that $R$ is a locally associated subring of $T$ while $S$ is not a locally associated subring of $T$. 
		To this end, let $T=\bQ[x]$ and $R=\bZ[x]$. Then for any nonzero $f(x)=a_0+a_1x+\dots+a_nx^n\in R$ (here, we may assume without loss of generality that $a_n\neq 0$) and $b\in\bZ$ such that $b\nmid a_n$, we have $\frac{1}{b}\cdot f\notin R$. Then $I=(R:T)=\{0\}$. Then as observed previously, $R$ is trivially a locally associated subring of $T$. However, letting $S=\bZ+x\bZ+x^2\bQ[x]$, we have that $R\subseteq S\subseteq T$ but, as shown above, $S$ is not a locally associated subring of $T$.
	\end{proof}
	
	Then in many cases that we are going to be concerned with, specifically integral ring extensions, any associated subring will also be ideal-preserving and locally associated. The natural question one might ask is whether the converse holds; that is, to show that a subring is associated, does it suffice to show that it is ideal-preserving and locally associated? This question is not yet answered in general, though we will show later that it holds in a particularly interesting and useful case.
	
	As we have alluded to, these subring relationships are of particular interest and utility when $R$ is an an order in a number field and $T=\Rbar$, the ring of algebraic integers in that number field. Thus, it will help to have more succinct terminology in this case.
	
	\begin{definition}
		Let $R$ be an order in a number field $K$.
		\begin{enumerate}
			\item If $R$ is an associated subring of $\Rbar$, we will say that $R$ is an \term{associated order}.
			\item If $R$ is an ideal-preserving subring of $\Rbar$, we will say that $R$ is an \term{ideal-preserving order}.
			\item If $R$ is a locally associated subring of $\Rbar$, we will say that $R$ is a \term{locally associated order}.
		\end{enumerate}
	\end{definition}
	
	\section{Equivalent Characterizations}
	To more fully understand these subrings and more easily check their conditions, we will explore additional alternate characterizations. As we will see, the more assumptions we make about the structures of $T$ and $R$, the more alternate characterizations we can produce.
	
	We will begin with characterizations of associated subrings.
	
	\begin{theorem}
		\label{as equiv conditions}
		Let $T$ be a commutative ring with identity, $R\subseteq T$ a subring (possibly without identity), and $I=(R:T)$. Consider the following condition:
		\begin{enumerate}
			\item $R$ is an associated subring of $T$.
		\end{enumerate}
		If $T$ is an integral domain, the following is equivalent to (1):
		\begin{enumerate}
			\item[(2)] For any $t\in T$ and collection of $T$-ideals $\{I_\alpha\}_{\alpha\in \Gamma}$,
			$$R\cap (t)\subseteq \bigcup_{\alpha\in \Gamma}I_\alpha \implies \:\exists\:\alpha\in \Gamma\st t\in I_\alpha.$$
		\end{enumerate}
		If $T$ is a Dedekind domain, the following is equivalent to (1):
		\begin{enumerate}
			\item[(3)] For any $t\in T$, if $(t)=P_1^{a_1}\dots P_k^{a_k}$ is the prime factorization of the principal ideal $(t)$ in $T$, then $$R\cap (t)\nsubseteq \left(\bigcup_{i=1}^kP_i^{a_i+1}\right)\cup\left(\bigcup_{Q\in\spec(T),Q\nmid (t)}Q\right).$$
		\end{enumerate}
		If $T$ is any commutative ring with identity and $1\in R$, the following is equivalent to (1):
		\begin{enumerate}
			\item[(4)] For any subset $\{u_\alpha\}_{\alpha\in \Gamma}\subseteq U(T)$ containing a representative from each coset of $\quot{U(T)}{U(R)}$ and $t\in T$, there exist $\alpha\in \Gamma$ and $\beta\in I$ such that $u_\alpha(t+\beta)\in R$. That is, in a slight abuse of notation, $\quot{T}{I}=\quot{R}{I}\cdot \quot{U(T)}{U(R)}$.
		\end{enumerate}
	\end{theorem}
	
	\begin{proof}
		Suppose that $T$ is an integral domain. In this case, we want to show $1\iff 2$. To that end, suppose first that $R$ is an associated subring of $T$. Let $t\in T$ and $\{I_\alpha\}_{\alpha\in\Gamma}$ a collection of $T$-ideals, and suppose that $R\cap (t)\subseteq \bigcup_{\alpha\in\Gamma}I_\alpha$. Since there is some $u\in U(T)$ such that $ut\in R\cap (t)$, this element is contained in the union of the $I_\alpha$'s. Then $ut\in I_\alpha$ for some $\alpha\in \Gamma$, so $t\in u^{-1}I_\alpha=I_\alpha$. Then $1\implies 2$.
		
		Now assume that Condition 2 holds. Then for some $t\in T$, let $\{(\alpha)\}_{\alpha\in R\cap (t)}$ be the set of all the principal $T$-ideals generated by an element of $R\cap (t)$. Since each element of $R\cap (t)$ is contained in the principal ideal it generates, $R\cap (t)\subseteq \bigcup_{\alpha\in R\cap (t)}(\alpha)$. Then by Condition 2, there is some $\alpha\in R\cap (t)$ such that $t\in (\alpha)$. Since $\alpha\in (t)$ by its definition, this tells us that $(t)=(\alpha)$. Since $T$ is an integral domain, this tells us that $\alpha$ is an associate of $t$, i.e. $t=u\alpha$ for some $u\in U(T)$. Since $\alpha\in R$, we can conclude that $R$ is an associated subring of $T$. Then $2\implies 1$ when $T$ is an integral domain.
		
		Now assume that $T$ is a Dedekind domain. We want to show $1\iff 3$. Since $T$ is an integral domain, we can do so by showing $2\iff 3$. First, suppose that Condition 2 holds. For any $t\in T$, we can factor the principal ideal $(t)$ into a product of primes $(t)=P_1^{a_1}\dots P_k^{a_k}$. Then note that for any $1\leq i\leq k$, $t\notin P_k^{a_k+1}$. Moreover, for any prime ideal $Q$ not in this factorization, $t\notin Q$. Then since $t$ is not in any of these ideals, it follows from Condition 2 that $R\cap (t)\nsubseteq \left(\bigcup_{i=1}^kP_i^{a_i+1}\right)\cup\left(\bigcup_{Q\in\spec(T),Q\nmid (x)}Q\right).$ Then $2\implies 3$.
		
		For the converse, assume that Condition 3 holds. Let $t\in T$ and $\{I_\alpha\}_{\alpha\in \Gamma}$ be a collection of $T$-ideals, none of which contains $t$. Since $t\notin I_\alpha$ for each $\alpha\in \Gamma$, then $I_\alpha\nmid (t)$; that is, each $I_\alpha$ must have a factor of either $P_i^{a_i+1}$ for some $1\leq i\leq k$ or $Q$ for some prime ideal $Q\nmid (t)$. Then $$R\cap (t)\nsubseteq \left(\bigcup_{i=1}^kP_i^{a_i+1}\right)\cup\left(\bigcup_{Q\in\spec(T),Q\nmid (x)}Q\right)\supseteq \bigcup_{\alpha\in \Gamma} I_\alpha,$$ so $R\cap (t)\nsubseteq \bigcup_{\alpha\in \Gamma} I_\alpha$. Then $3\implies 2$.
		
		Finally, we will relax the assumptions we have previously made about $T$ and assume that $1\in R$. In this case, we want to show $1\iff 4$. First, assume that $R$ is an associated subring of $T$. Let $\{u_\alpha\}_{\alpha\in \Gamma}\subseteq U(T)$ be a subset containing a representative from each coset of $\quot{U(T)}{U(R)}$, and let $t\in T$. Since $R$ is an associated subring of $T$, there is some $u\in U(T)$ such that $ut=r\in R$. Then letting $\alpha\in \Gamma$ such that $u_\alpha$ is a representative of the coset $u\cdot U(R)$, there exists some $v\in U(R)$ such that $u_\alpha=uv$. Then $u_\alpha t=uvt=vr\in R$. Finally, letting $\beta=0\in I$, we have $u_\alpha(t+\beta)\in R$. Then $1\implies 4$.
		
		For the converse, assume that condition 4 holds and let $t\in T$. Then since $U(T)$ is itself a subset of $U(T)$ containing a representative from each coset of $\quot{U(T)}{U(R)}$, there is some $u\in U(T)$ and $\beta\in I$ such that $u(t+\beta)=r\in R$. Then $ut=r-u\beta\in R$, and thus $R$ is an associated subring of $T$. Then $4\implies 1$, completing the proof.
	\end{proof}
	
	There a few important notes one can make about these characterizations of associated subrings. First, Condition 2 actually provides a purely ideal-theoretic characterization of the property (this can be seen more clearly if we refer to principal ideals $(t)\subseteq T$ rather than elements $t\in T$). Furthermore, Condition 4 is in general a much easier condition to check, requiring us only to consider cosets in $\quot{T}{I}$, $\quot{R}{I}$, and $\quot{U(T)}{U(R)}$ rather than all elements in $T$, $R$, and $U(T)$. In the case when $R$ is an order in a number field and $T=\Rbar$, which we will return to later on, this actually provides a characterization which can always be checked in finitely many steps. This is because in this case, $\quot{T}{I}$, $\quot{R}{I}$, and $\quot{U(T)}{U(R)}$ will all be finite.
	
	\begin{theorem}
		\label{ip equiv conditions}
		Let $T$ be a commutative ring, $R\subseteq T$ a subring (neither assumed to have identity), and $I=(R:T)$. Consider the following condition:
		\begin{enumerate}
			\item $R$ is an ideal-preserving subring of $T$.
		\end{enumerate}
		If $T$ is an integral domain, the following is equivalent to (1):
		\begin{enumerate}
			\item[(2)] For any $T$-ideal $J$ and $t\in T$, $$t\notin J\implies R\cap (t)\nsubseteq J.$$
		\end{enumerate}
		If $T$ is a Dedekind domain, the following is equivalent to (1):
		\begin{enumerate}
			\item[(3)] If $P_1,P_2$ are prime $T$-ideals, then $R\cap P_1\nsubseteq P_1^2$ and, if $P_1\neq P_2$, then $R\cap P_1\nsubseteq P_2$.
		\end{enumerate}
		If $T$ is a Dedekind domain and $1\in R$, the following is equivalent to (1):
		\begin{enumerate}
			\item[(4)] If $P_1,P_2$ are prime $T$-ideals containing $I$, then $R\cap P_1\nsubseteq P_1^2$, and if $P_1\neq P_2$, then $R\cap P_1\nsubseteq P_2$.
		\end{enumerate}
		If $T$ is a Dedekind domain, $1\in R$, and $T$ is integral over $R$, then the following is equivalent to (1):
		\begin{enumerate}
			\item[(5)] If $I=P_1^{a_1}\dots P_k^{a_k}$ is the factorization of $I$ into prime $T$-ideals, then $R\cap P_i\nsubseteq P_i^2$ for $1\leq i\leq k$ and 
			$$\quot{R}{I}\cong \prod_{i=1}^k\quot{R}{R\cap P_i^{a_i}}\cong \prod_{i=1}^k\quot{R+P_i^{a_i}}{P_i^{a_i}}.$$
		\end{enumerate}
	\end{theorem}
	
	\begin{proof}
		Suppose that $T$ is an integral domain. In this case, we need to show $1\iff 2$. Assume that Condition 1 holds and let $t\in T$ and $J$ be a $T$-ideal. Then note that $t\notin J$ is equivalent to the statement $(t)\nsubseteq J$. Since $R$ is an ideal-preserving subring, this would imply that $R\cap (t)\nsubseteq J$. Then $1\implies 2$.
		
		Now assume that Condition 2 holds and let $J_1, J_2$ be $T$-ideals such that $J_1\nsubseteq J_2$. Then there is some $t\in J_1\backslash J_2$. Since $t\notin J_2$, Condition 2 tells us that $R\cap (t)\nsubseteq J_2$. Then since $R\cap (t)\subseteq R\cap J_1$, it follows that $R\cap J_1\nsubseteq J_2$. Then $R$ is an ideal-preserving subring of $T$, so $2\implies 1$.
		
		Now suppose that $T$ is a Dedekind domain; we want to show that $1\iff 3$. It is immediately clear in this case that $1\implies 3$, since the statement of $3$ is simply an application of the definition of an ideal-preserving subring to the prime ideals in $T$. We need to show the converse. To that end, let $J_1\nsubseteq J_2$ be $T$-ideals. Since $T$ is a Dedekind domain, this is equivalent to the statement $J_2\nmid J_1$. Then letting $J_1=P_1^{a_1}\dots P_k^{a_k}$ and $J_2=P_1^{b_1}\dots P_k^{b_k}$ be the representations of $J_1$ and $J_2$ as products of prime $T$-ideals (allowing the possibility that some $a_i$ or $b_j$ is 0), it must be the case that for some $1\leq i\leq k$, $b_i>a_i$. By Condition 3, we can select some element $\alpha_i\in R\cap P_i\backslash P_i^2$ and, for each $1\leq j\leq k$, $j\neq i$, we can select $\alpha_j\in R\cap P_j\backslash P_i$. Then the element $\alpha=\alpha_1^{a_1}\dots \alpha_k^{a_k}\in R\cap (P_1^{a_1}\dots P_k^{a_k})=R\cap J_1$. However, by the choice of the $\alpha_j$'s, $P_i^{a_i}$ is the exact power of $P_i$ dividing $(\alpha)$. Then since $b_i>a_i$, $\alpha\notin P_i^{b_i}\supseteq J_2$, so $\alpha\in R\cap J_1\backslash J_2$. Therefore, $R$ is an ideal-preserving subring of $T$, so $3\implies 1$.
		
		Next, suppose that $T$ is a Dedekind domain and $1\in R$. We want to show $1\iff 4$; since $T$ is a Dedekind domain, it will suffice to show $3\iff 4$. Again, $3\implies 4$ trivially, since we are only limiting the prime $T$-ideals we are considering. We need only show the converse, $4\implies 3$. To do so, we will show that Condition 3 will always hold if either $P_1$ or $P_2$ (or both) do not divide the conductor ideal. It is worth noting at this point that we will be using the convention that any ideal divides (contains) the zero ideal; thus, if $I=\{0\}$, Conditions 3 and 4 are exactly the same.
		
		First, suppose that $P_1$ is a prime $T$-ideal such that $P_1\nmid I$ (if $I$ has a prime divisor; if not, then $T=R=I$ and $R$ is trivially an ideal-preserving subring of $T$). Then letting $\alpha\in P_1\backslash P_1^2$ and $\beta\in I\backslash P_1$, we have $\alpha\beta\in R\cap P_1\backslash P_1^2$, so $R\cap P_1\nsubseteq P_1^2$. Now let $P_1$ and $P_2$ be distinct prime $T$-ideals. If $P_2\nmid I$, then let $\alpha\in P_1\backslash P_2$ and $\beta\in I\backslash P_2$. Then $\alpha\beta\in R\cap P_1\backslash P_2$, so $R\cap P_1\nsubseteq P_2$. 
		
		The only remaining case to consider is when $P_1\nmid I$ and $P_2|I$. In this case, note that $P_1$ and $I$ are relatively prime ideals in $T$, and thus there exist $\alpha\in P_1$ and $\beta\in I$ such that $1=\alpha+\beta$. Then note that $\alpha=1-\beta\in P_1$. Since $\beta\in I\subseteq P_2$, $\alpha=1-\beta\notin P_2$. Finally, since $1\in R$ and $\beta\in I\subseteq R$, $\alpha=1-\beta\in R$. Then $\alpha=1-\beta\in R\cap P_1\backslash P_2$, so $R\cap P_1\nsubseteq P_2$. Then $4\implies 3$.
		
		Finally, assume that $T$ is a Dedekind domain, $1\in R$, and $T$ is integral over $R$. We want to show $1\iff 5$. To start, we will note that the second isomorphism in the statement of Condition 5 will always hold, regardless of whether $R$ is an ideal-preserving subring of $T$. This can easily be seen by considering the obvious mapping $\phi:\prod_{i=1}^k\quot{R}{R\cap P_i^{a_i}}\to \prod_{i=1}^k\quot{R+P_i^{a_i}}{P_i^{a_i}}$ defined by $\phi(r_1+R\cap P_1^{a_1},\dots,r_k+R\cap P_k^{a_k})=(r_1+P_1^{a_1},\dots,r_k+P_k^{a_k})$. We will focus on the first isomorphism.
		
		Assume that $R$ is an ideal-preserving subring of $T$. This immediately tells us that for each $P_i$ dividing $I$, we have $R\cap P_i\nsubseteq P_i^2$. Now let $\phi:\quot{R}{I}\to \prod_{i=1}^k\quot{R}{R\cap P_i^{a_i}}$ be the mapping defined by $\phi(r+I)=(r+R\cap P_1^{a_1},\dots,r+R\cap P_k^{a_k})$. It is plain to see that $\phi$ is a well-defined homomorphism; moreover, $\phi(r+I)=0\implies r\in \bigcap_{i=1}^k R\cap P_i^{a_i}=I$. Then $\ker(\phi)=\{0\}$, so $\phi$ is injective. All that remains is to show that $\phi$ is surjective. To show this, we will show that for each $1\leq i\leq k$, there exists $r_i\in R$ such that $\phi(r_i+I)$ is 0 in every coordinate except the $i^{th}$, in which it is $1+P_i^{a_i}$; that is, $r_i\equiv 1\modulo{P_i^{a_i}}$ and $r_i\in P_j^{a_j}$ for each $j\neq i$.
		
		Since $R$ is an ideal-preserving subring of $T$, we can select some $x_i\in R\cap \prod_{j\neq i}P_j^{a_j}\backslash P_i$. Now since $x_i\notin P_i$, we know that $x_i+I\in U(\sfrac{T}{P_i^{a_i}})$. Let $S\subseteq \quot{T}{P_i^{a_i}}$ be the subset whose cosets contain elements of $R$. Since $R$ is a subring of $T$ and $T$ is integral over $R$, it is plain to see that $S$ is a subring of $\quot{T}{P_i^{a_i}}$ and $\quot{T}{P_i^{a_i}}$ is integral over $S$. Then any element of $U(\sfrac{T}{P_i^{a_i}})$ which lies in $S$ has its inverse lying in $S$; in particular, this means that there exists $y_i\in R$ such that $x_iy_i+P_i^{a_i}=1+P_i^{a_i}$. Then $r_i:=x_iy_i\equiv 1\modulo{P_i^{a_i}}$ and $r_i\in P_j^{a_j}$ for every $j\neq i$. Then $\phi$ is onto and is thus an isomorphism. Therefore, $1\implies 5$.
		
		All that remains is to show that $5\implies 1$; to do so, we will show that $5\implies 4$. Since 5 makes the assumption that for any $P_i|I$, $R\cap P_i\nsubseteq P_i^2$, we only need to show that for any prime $T$-ideals $P_i,P_j$ dividing $I$, $R\cap P_i\nsubseteq P_j$. Using the isomorphism $\phi$ developed above, there must be some $r\in R$ such that $r\in P_i^{a_i}$ and $r\equiv 1\modulo{P_j^{a_j}}$. Then $r\in R\cap P_i\backslash P_j$, so $R\cap P_i\nsubseteq P_j$. Thus, $5\implies 4$, completing the proof. 
	\end{proof}
	
	Again, we can make a few notes about these characterizations. From Condition 2, we can see that when working in an integral domain, it suffices to check the ideal-preserving condition when $J_1$ is a principal ideal. Comparing this with Condition 2 for associated subrings, we note that the only difference (after applying the contrapositive) is that in an associated subring, the principal ideal $(t)$ remains separate from any arbitrary union of ideals when restricted to $R$, rather than just a single ideal. Thus, associated subrings can, in a sense, be thought of as ``strongly" ideal-preserving. Another thing we might observe is that Conditions 3 and 4 greatly reduce the scope of what needs to be checked to verify the ideal-preserving condition. In fact, if the conductor ideal $I$ is nonzero, Condition 4 can be checked in finitely many steps. Again, this will always be the case when $R$ is an order in a number field and $T=\Rbar$.
	
	To produce a similar set of characterizations for locally associated subrings, we need the following lemma from \cite{neukirch}. For a detailed proof, see \cite{thesis}.
	
	\begin{lemma}
		\label{exact sequence}
		Let $R$ be an order in a number field $K$ with conductor ideal $I$. Then there is an exact sequence
		$$1\to U(R)\to U(\Rbar)\times U(\quot{R}{I})\to U(\quot{\Rbar}{I})\to \Cl(R)\to \Cl(\Rbar)\to 1.$$
		Thus, the class numbers $\abs{\Cl(R)}$ and $\abs{\Cl(\Rbar)}$ are related as follows:
		$$\abs{\Cl(R)}=\abs{\Cl(\Rbar)}\cdot \frac{\abs{U(\quot{\Rbar}{I})}}{\abs{U(\quot{R}{I})}\cdot \abs{\quot{U(\Rbar)}{U(R)}}}.$$
	\end{lemma}
	
	\begin{theorem}
		\label{la equiv conditions}
		Let $T$ be a commutative ring with identity, $R\subseteq T$ a subring with identity, and $I=(R:T)$. The following are equivalent:
		\begin{enumerate}
			\item $R$ is a locally associated subring of $T$.
			\item Every coset in $\quot{U(\sfrac{T}{I})}{U(\sfrac{R}{I})}$ contains a unit in $T$; that is, for any $t+I\in U(\sfrac{T}{I})$, there exist $r+I\in U(\sfrac{R}{I})$ and $\beta\in I$ such that $tr+\beta\in U(T)$.
			\item If $t\in T$ is relatively prime to $I$, i.e. $tT+I=T$, then there exists $r\in R$ relatively prime to $I$, i.e. $rR+I=R$, and $u\in U(T)$ such that $t=ru$.
		\end{enumerate}
		If $R$ is an order in a number field and $T=\Rbar$, the following are equivalent to (1):
		\begin{enumerate}
			\item[(4)] $\abs{\quot{U(T)}{U(R)}}=\frac{\abs{U(\quot{T}{I})}}{\abs{U(\quot{R}{I})}}$.
			\item[(5)] $\abs{\Cl(T)}=\abs{\Cl(R)}$.
			\item[(6)] $\Cl(T)\cong\Cl(R)$.
		\end{enumerate}
	\end{theorem}
	
	\begin{proof}
		We have already seen from Proposition \ref{la first conditions} that the first three conditions are equivalent. We just need to show the same for Conditions 4, 5, and 6.
		
		Suppose that $R$ is an order in a number field and $T=\Rbar$. In this case, we know that $U(\sfrac{T}{I})$, $U(\sfrac{R}{I})$, and $\quot{U(\Rbar)}{U(R)}$ are finite groups. Now recall the injective map $\phi:\quot{U(T)}{U(R)}\to \quot{U(\sfrac{T}{I})}{U(\sfrac{R}{I})}$ discussed in the proof of Proposition \ref{la first conditions}. Since this is an injective map from a finite group into a finite group, $\phi$ is an isomorphism if and only if its domain and codomain have the same order. That is, $1\iff 4$. 
		
		The equivalence $4\iff 5$ follows immediately from the lemma, since $\abs{\Cl(T)}=\abs{\Cl(R)}$ if and only if the fraction multiplier is equal to 1. Finally, since the exact sequence presented in the lemma states that there is a surjective homomorphism from $\Cl(R)$ onto $\Cl(T)$, both finite groups, it follows that these groups are isomorphic if and only if they have the same order. Then $5\iff 6$, completing the proof.
	\end{proof}
	
	These characterizations of locally associated subrings (in particular, those that hold in the case when $R$ is an order in a number field and $T=\Rbar$) provide not only a purely ideal-theoretic characterization in Condition 6, but also a characterization that can be more easily checked in finitely many steps in Condition 4.
	
	\section{Related Rings}
	With these equivalent characterizations of associated, ideal-preserving, and locally associated subrings in hand, we will now explore how these properties can allow us to understand rings related to $T$ and $R$. First, we consider the rings of polynomials and formal power series over $T$ and $R$.
	
	\begin{theorem}
		\label{la ps and poly}
		Let $T$ be an integral domain, $R$ a subring of $T$ with identity, and $I=(R:T)$. If $R$ is a locally associated subring of $T$, then $R[x]$ is a locally associated subring of $T[x]$ and $R[[x]]$ is a locally associated subring of $T[[x]]$. 
	\end{theorem}
	
	\begin{proof}
		It is not difficult to see that $I[x]=(R[x]:T[x])$ and $I[[x]]=(R[[x]]:T[[x]])$. We will start by showing that $R[x]$ is a locally associated subring of $T[x]$. Using the second characterization of locally associated subrings from Theorem \ref{la equiv conditions}, we can do so by showing that for any $f+I[x]\in U(\sfrac{T[x]}{I[x]})$, there exist $g+I[x]\in U(\sfrac{R[x]}{I[x]})$ and $h\in I[x]$ such that $fg+h\in U(T[x])$.
		
		Recall that for any integral domain $S$, $U(S[x])=U(S)$. Then $f(x)=a_0+a_1x+\dots+a_nx^n\in T[x]$ is a unit modulo $I[x]$ if and only if $a_0+I\in U(\sfrac{T}{I})$ and $a_i\in I$ for $1\leq i\leq n$. Since $R$ is a locally associated subring of $T$, there must therefore be some $r+I\in U(\sfrac{R}{I})$ and $\beta\in I$ such that $a_0r+\beta=u\in U(T)$. Then letting $g(x)=r$ (with $g+I[x]\in U(\sfrac{T[x]}{I[x]}))$ and $h(x)=\beta-ra_1x-\dots-ra_nx^n\in I[x]$, we have that $fg+h=u\in U(T[x])$. Then $R[x]$ is a locally associated subring of $T[x]$.
		
		The argument for $R[[x]]$ and $T[[x]]$ is similar; in this case, recall that for any commutative ring with identity $S$, $U(S[[x]])=U(S)+xS[[x]]$. Then $f(x)=a_0+a_1x+a_2x^2+\dots\in T[[x]]$ is a unit modulo $I[[x]]$ if and only if $a_0+I\in U(\sfrac{T}{I})$. Since $R$ is a locally associated subring of $T$, there must be some $r+ I\in U(\sfrac{R}{I})$ and $\beta\in I$ such that $a_0r+\beta\in U(T)$. Then $fr+\beta=u+a_1rx+a_2rx^2+\dots\in U(T[[x]])$, with $r+I[[x]]\in U(\sfrac{R[[x]]}{I[[x]]})$ and $\beta\in I[[x]]$. Then $R[[x]]$ is a locally associated subring of $T[[x]]$.
	\end{proof}
	
	This same sort of inheritance can be observed in some specific cases for ideal-preserving and associated subrings as well. For instance, since $\bQ$ is a field, $\bZ$ is an associated (and thus ideal-preserving) subring of $\bQ$. As we have previously noted, $\bZ[x]$ is an associated (and thus ideal-preserving) subring of $\bQ[x]$ and $\bZ[[x]]$ is an associated (and thus ideal-preserving) subring of $\bQ[[x]]$. However, this type of inheritance is not seen in general; in fact, inheritance of the associated subring property in polynomial rings will never be seen if $T$ is integral over its proper subring $R$, as the following result shows.
	
	\begin{theorem}
		\label{poly not associated}
		Let $T$ be an integral domain and $R$ an associated subring (with identity) of $T$. If either $U(T)\cap R=U(R)$ or $\quot{U(T)}{U(R)}$ is finite, then $R[x]$ is an associated subring of $T[x]$ if and only if $R=T$.
	\end{theorem}
	
	\begin{proof}
		Suppose first that $U(T)\cap R=U(R)$. Also assume that $R[x]$ is an associated subring of $T[x]$. Then for any $t\in T$, there must be some $u\in U(T[x])=U(T)$ such that $u(1+tx)=u+(ut)x\in R[x]$. Note in this case that $u\in R\cap U(T)=U(R)$ and $ut\in R$. Then letting $ut=r\in R$, $t=u^{-1}r\in R$, so $T\subseteq R$. Then if $R[x]$ is an associated subring of $T[x]$, it must be the case that $R=T$; the reverse implication is obvious.
		
		Now assume that $\quot{U(T)}{U(R)}$ is a finite quotient group and $R[x]$ is an associated subring of $T[x]$. Denote by $\{u_i\}_{i=0}^k\subseteq U(T)$ a list of coset representatives of $\quot{U(T)}{U(R)}$, one from each coset. Then for any $t\in T$, there must exist some $u\in U(T[x])=U(T)$ such that $u(u_0^{-1}t+u_1^{-1}tx+u_2^{-1}tx^2+\dots+u_k^{-1}tx^k)\in R[x]$, i.e. $uu_i^{-1}t\in R$ for each $0\leq i\leq k$. Since each coset of $\quot{U(T)}{U(R)}$ can be represented by one of the $u_i$, there must be some $0\leq i\leq k$ such that $uu_i^{-1}\in R$. Then letting $uu_i^{-1}t=r\in R$, we have $t=u^{-1}u_ir\in R$, so $T\subseteq R$. Again, if $R[x]$ is an associated subring of $T[x]$, it must be the case that $R=T$; the reverse implication is again obvious.
	\end{proof}
	
	One will note that the problem with polynomial ring extensions inheriting the associated subring property is that the unit group is still limited to units in the base ring. For rings of formal power series, the unit group is much larger, allowing the associated subring property to be much more easily inherited. The following theorem from \cite{radicalconductor} shows one such case in which inheritance of the associated subring condition is observed; the example afterward shows that this inheritance for associated and ideal-preserving subrings is not observed in general.
	
	\begin{theorem}
		Let $R$ be an associated order in a number field $K$ with radical conductor ideal $I$. Then $R[[x]]$ is an associated subring of $\Rbar[[x]]$.
	\end{theorem}
	
	\begin{example}
		Let $K=\bQ[\alpha]$, where $\alpha\in\bC$ is a root of $f(x)=x^3+4x-1$. In \cite{radicalconductor}, it was shown that the ring $R=\bZ+9\bZ+(2-4\alpha+\alpha^2)\bZ$ is an order in $K$ which is an associated subring of $\Rbar=\bZ[\alpha]$. Then $R$ is both an ideal-preserving and locally associated subring of $\Rbar$. By Theorem \ref{la ps and poly}, this tells us that $R[[x]]$ is a locally associated subring of $\Rbar[[x]]$ and $R[x]$ is a locally associated subring of $\Rbar[x]$. However, \cite{radicalconductor} also showed that $R[[x]]$ is not an associated subring of $\Rbar[[x]]$, and Theorem \ref{poly not associated} shows that $R[x]$ is not an associated subring of $\Rbar[[x]]$.
		
		Now consider the ideal $J=(3+\alpha x)\subseteq \Rbar[x]$, and let $f=a_0+a_1x+\dots+a_nx^n\in \Rbar[x]$ such that $(3+\alpha x)f\in R[x]$ (i.e. $(3+\alpha x)f$ is an arbitrary element of $J\cap R[x]$). Looking at the constant and linear terms of this product, $3a_0\in R$ and $3a_1+\alpha a_0\in R$. Let $a_0=x_0+y_0\alpha+z_0(2-4\alpha+\alpha^2)$ and $a_1=x_1+y_1\alpha+z_1(2-4\alpha+\alpha^2)$. Since $3a_0\in R$, then $3|y_0$. Furthermore, $3a_1+\alpha a_0=(3x_1-2y_0+9z_0)+(3y_1+x_0+4y_0-18z_0)\alpha+(3z_1+y_0-4z_0)(2-4\alpha+\alpha^2)\in R$. Since $3|y_0$, then $3|x_0$. Then $a_0\in 3\bZ+3\alpha\bZ+(2-4\alpha+\alpha^2)\bZ\subseteq (3,2+2\alpha+\alpha^2)$. Letting $P=(3,2+2\alpha+\alpha^2)$, this tells us that $(3+\alpha x)f$ has a constant term lying in the ideal $3P$, with $3\notin 3P$. Then $J\nsubseteq 3P+x\Rbar[x]$, but $J\cap R[x]\subseteq 3P+x\Rbar[x]$, meaning that $R[x]$ is not an ideal-preserving subring of $\Rbar[x]$. The same argument also tells us that $R[[x]]$ is not an ideal-preserving subring of $\Rbar[[x]]$.
	\end{example}
	
	In addition to these results, we can often determine a great deal of information about intermediate subrings; that is, for a ring $T$ with subring $R\subseteq T$, we can often find properties of a ring $S$ such that $R\subseteq S\subseteq T$. This will in particular be very helpful when examining orders in a number field.
	
	\begin{theorem}
		\label{ip intermediate cond}
		Let $T$ be a Dedekind domain, $R$ an ideal-preserving subring of $T$ with identity such that $I=(R:T)\neq \{0\}$. Then for any $T$-ideal $J$, $(R+J:T)=I+J$. In particular, if $I\subseteq J$, then $(R+J:T)=J$.
	\end{theorem}
	
	\begin{proof}
		First, note that $R+J$ is a ring lying between $R$ and $T$, i.e. $R\subseteq R+J\subseteq T$. Furthermore, $I+J$ is an ideal containing $I$ with $R+(I+J)=R+J$. Then it will suffice to show the desired result for ideals containing $I$. Also note that if $J=\{0\}$, then the result holds trivially; thus, we will assume moving forward that $J$ is a nonzero ideal.
		
		Note that $J\subseteq R+J$. Then letting $A=(R+J:T)$, it must be the case that $J\subseteq A$; suppose that $J\subsetneq A$. In this case, $R+J\subseteq R+A\subseteq R+(R+J)=R+J$, so $R+J=R+A$. Since $T$ is a Dedekind domain, we can now factor $J$ and $A$ into prime $T$-ideals, $A=P_1^{a_1}\dots P_k^{a_k}$, $J=P_1^{b_1}\dots P_k^{b_k}$, and $I=P_1^{c_1}\dots P_k^{c_k}$. Since $I\subseteq J\subsetneq A$, then $a_j\leq b_j\leq c_j$ for each $1\leq j\leq k$ and there is at least one $i$ such that $a_i<b_i$. Without loss of generality, assume $a_1<b_1$.
		
		Since $R$ is an ideal-preserving subring of $T$, we can now construct elements as follows: let $\alpha\in R\cap A\backslash P_1^{a_1+1}$; let $\beta_1\in R\cap P_1 \backslash P_1^2$; and for $2\leq j\leq k$, let $\beta_j\in R\cap P_j\backslash P_1$. Finally, let $\beta=\beta_1^{c_1-b_1}\dots\beta_k^{c_k-b_k}$. Since $\alpha\in A$, $\alpha T\subseteq R+J$; then for any $t\in T$, there is some $r\in R$ and $\gamma\in J$ such that $\alpha t=r+\gamma$. Therefore, $(\alpha\beta)t=r\beta+\gamma\beta$. Since $\beta\in R$, $r\beta\in R$; From the prime factors dividing $J$ and $\beta$, $\gamma\beta\in I\subseteq R$. Then for any $t\in T$, $(\alpha\beta)t\in R$. Then by definition, $\alpha\beta\in (R:T)=I$. However, the exact power of $P_i$ dividing $(\alpha\beta)$ is $P_i^{a_i+c_i-b_i}=P_i^{c_i-(b_i-a_i)}$. Since $b_i>a_i$, this means that $P_i^{c_i}\nmid (\alpha\beta)$, so $\alpha\beta\notin I$, a contradiction. Then it must be the case that $A=J$.
	\end{proof}
	
	This result allows us to use the fact that $R$ is an ideal-preserving subring of $T$ to give us information about $R+J$, an ring lying between $R$ and $T$. As we have seen previously, we also know that such $R+J$ will necessarily be an ideal-preserving subring of $T$. Then when $T$ is a Dedekind domain, an ideal-preserving subring $R$ creates, in a sense, a network of ideal-preserving subrings intermediate to the ring extension $R\subseteq T$. This network has several interesting and useful properties that we will explore, especially in the case when $R$ is an ideal-preserving order in a number field.
	
	The previous result gives us a way to move from a smaller ring to a larger ring within this network; the next will show what happens when we take the intersection of two rings in this network.
	
	\begin{theorem}
		\label{ip intersect cond ideals}
		Let $T$ be a Dedekind domain and $R$ an ideal-preserving subring of $T$ with identity such that $I=(R:T)\neq \{0\}$. Let $J_1$ and $J_2$ be two $T$-ideals containing $I$ and define $R_1:=R+J_1$ and $R_2:=R+J_2$. Then $R_1\cap R_2$ is a subring of $T$ such that $R\subseteq R+J_1\cap J_2\subseteq R_1\cap R_2$ and $(R_1\cap R_2:T)=J_1\cap J_2$. Moreover, if $T$ is integral over $R$, then $R_1\cap R_2=R+J_1\cap J_2$.
	\end{theorem}
	
	\begin{proof}
		First, note that $R$ and $J_1\cap J_2$ are trivially contained in $R_1\cap R_2$ by the definition of $R_1$ and $R_2$, so $R\subseteq R+J_1\cap J_2\subseteq R_1\cap R_2$. 
		
		To show that $(R_1\cap R_2:T)=J_1\cap J_2$, note as above that $J_1\cap J_2\subseteq R_1\cap R_2$. Then $J_1\cap J_2\subseteq (R_1\cap R_2:T)$. For the reverse inclusion, let $\alpha\in (R_1\cap R_2:T)$; in other words, $\alpha T\subseteq R_1\cap R_2$. Since $\alpha T\subseteq R_1$, then $\alpha \in J_1=(R_1:T)$; similarly, $\alpha T\subseteq R_2$ so $\alpha\in J_1=(R_2:T)$. Then $\alpha\in J_1\cap J_2$, so $(R_1\cap R_2:T)=J_1\cap J_2$.
		
		Now assume that $T$ is integral over $R$. We want to show that $R_1\cap R_2=R+J_1\cap J_2$. Since we have already shown that $R+J_1\cap J_2\subseteq R_1\cap R_2$ holds more generally, we only need to show the reverse inclusion. To start, note that $R+J_1\cap J_2$ is a subring of $T$ containing 1 with conductor ideal $(R+J_1\cap J_2:T)=J_1\cap J_2$, $T$ is integral over $R+J_1\cap J_2$, $R_1=R+J_1=(R+J_1\cap J_2)+J_1$, and $R_2=R+J_2=(R+J_1\cap J_2)+J_2$. Then it will suffice to show that $R_1\cap R_2=R$ in the case that $J_1\cap J_2=I$.
		
		Assume that $J_1\cap J_2=I$; for now, we will also assume that $J_1$ and $J_2$ are relatively prime. By applying Condition 5 in Theorem \ref{ip equiv conditions} first to $\quot{R}{I}$ and then to $\quot{R_1}{J_1}$ and $\quot{R_2}{J_2}$, we will get that $\quot{R}{I}\cong \quot{R_1}{J_1}\times \quot{R_2}{J_2}$. Then let $t\in R_1\cap R_2$. By the isomorphism produced in the proof of Theorem \ref{ip equiv conditions}, there must be some $r\in R$ such that $(r+J_1,r+J_2)=(t+J_1,t+J_2)$. Then $t-r\in J_1\cap J_2=I$, so $t=r+\beta$ for some $\beta\in I$ and thus $t\in R$. Then when $J_1$ and $J_2$ are relatively prime, $R_1\cap R_2=R$.
		
		Now to show that this holds even when $J_1$ and $J_2$ are not relatively prime, write $I=P_1^{a_1}\dots P_k^{a_k}$, the factorization of $I$ into prime $T$-ideals. Since $J_1\cap J_2=I$, then it must be the case that $J_1=P_1^{b_1}\dots P_k^{b_k}$ and $J_2=P_1^{c_1}\dots P_k^{c_k}$ with $b_i,c_i\in\bN_0$ and $\max\{b_i,c_i\}=a_i$ for each $1\leq i\leq n$. Since $R_1=R+J_1\subseteq R+P_i^{b_i}$ and $R_2=R+J_2\subseteq R+P_i^{c_i}$ for each $1\leq i\leq k$, then $R_1\cap R_2\subseteq \bigcap_{i=1}^n(R+P_i^{a_i})$. Applying the relatively prime case $n$ times then gives us $R_1\cap R_2\subseteq \bigcap_{i=1}^n(R+P_i^{a_i})=R+I=R$, completing the proof.
	\end{proof}
	
	Note in this theorem that, although we are intersecting two larger rings $R_1$ and $R_2$ and examining the resulting smaller ring, we are still depending on the fact that $R_1$ and $R_2$ were produced in a very particular manner from a common subring $R$. The following shows how we might carry out this same process without starting from a common subring $R$ (even if the subrings in question are not ideal-preserving).
	
	\begin{theorem}
		Let $R_1$ and $R_2$ be two subrings of the same commutative ring with identity $T$, and let $J_1:=(R_1:T)$ and $J_2:=(R_2:T)$. Then $R:=R_1\cap R_2$ is a subring of $T$ with conductor ideal $I:=(R:T)=J_1\cap J_2$. Moreover, if $T$ is a Dedekind domain, $1\in R$, $T$ is integral over $R$, and $J_1$ and $J_2$ are relatively prime as $T$-ideals, then $R_1=R+J_1$, $R_2=R+J_2$, and $$\quot{R}{I}\cong\quot{R_1}{J_1}\times \quot{R_2}{J_2}.$$ 
	\end{theorem}
	
	\begin{proof}
		Define $R=R_1\cap R_2$ and $I=(R:T)$. Then for any $\alpha\in I$, $\alpha T\subseteq R$. Then since $R\subseteq R_1$ and $R\subseteq R_2$, $\alpha T\subseteq R_1\cap R_2$, so $\alpha\in J_1\cap J_2$. Then $I\subseteq J_1\cap J_2$.
		
		Now let $\alpha\in J_1\cap J_2$, and let $t\in T$. Since $\alpha\in J_1$, we have $\alpha t\in R_1$. Similarly, since $\alpha\in J_2$, $\alpha t\in R_2$. Then $\alpha T\subseteq R_1\cap R_2=R$, so $\alpha\in I$. Thus, $J_1\cap J_2\subseteq I$, so $I=J_1\cap J_2$.
		
		Now assume that $T$ is a Dedekind domain, $1\in R$, $T$ is integral over $R$, and $J_1$ and $J_2$ are relatively prime as $T$-ideals. Consider the map $\phi:\quot{R}{I}\to \quot{R_1}{J_1}\times \quot{R_2}{J_2}$ defined by $\phi(r+I)=(r+J_1,r+J_2)$. Since $R=R_1\cap R_2$ and $I=J_1\cap J_2$, this is a well-defined homomorphism. Furthermore, $\phi(r+I)=0$ if and only if $r\in J_1\cap J_2=I$, so $\ker(\phi)$ is trivial and thus $\phi$ is injective. We need to show that $\phi$ is surjective as well. Let $r_1\in R_1$ and $r_2\in R_2$. By the Chinese Remainder Theorem applied to $\quot{T}{I}$, there exists $r\in T$ such that $r\equiv r_1\modulo{J_1}$ and $r\equiv r_2\modulo{J_2}$. Then for some $\beta_1\in J_1$ and $\beta_2\in J_2$, $r=r_1+\beta_1\in R_1$ and $r=r_2+\beta_2\in R_2$. Then $r\in R_1\cap R_2=R$, so $\phi(r+I)=(r_1+J_1,r_2+J_2)$. Then $\phi$ is onto, which gives the desired isomorphism. Note that this also tells us that $r_1=r-\beta_1\in R+J_1$ and $r_2=r-\beta_2\in R+J_2$, so $R_1\subseteq R+J_1$ and $R_2\subseteq R+J_2$. The reverse inclusions are immediate, so $R_1=R+J_1$ and $R_2=R+J_2$. 
	\end{proof}
	
	These theorems introduce the idea of intersecting two subrings to produce a smaller subring; a natural question, then, is whether any of our subring relationships are inherited by this smaller ring. As the next result shows, ideal-preserving subrings are ``nice" in the sense that the property is easily inherited in this way.
	
	\begin{theorem}
		Let $R_1$ and $R_2$ be two-ideal preserving subrings with identity of the same Dedekind domain $T$, and let $J_1=(R_1:T)$ and $J_2=(R_2:T)$ be relatively prime as $T$-ideals. Then $R=R_1\cap R_2$ is an ideal-preserving subring of $T$.
	\end{theorem}
	
	\begin{proof}
		First, note by the previous theorem that $R$ is a subring of $T$ with conductor ideal $I=(R:T)=J_1\cap J_2=J_1J_2$. First, assume that $J_1=0$. Since $J_1+J_2=T$, then it must be the case that $J_2=R_2=T$. Then $R=R_1\cap R_2=R_1$, which is an ideal-preserving subring of $T$. By symmetry, the same argument holds if $J_2=0$. From here, we will assume that $J_1$ and $J_2$ are relatively prime nonzero $T$-ideals. Thus, we can factor $J_1$ and $J_2$ into products of prime $T$-ideals, and any prime dividing $J_1$ cannot divide $J_2$ (and vice versa). Condition 4 from Theorem \ref{ip equiv conditions} tells us that to show that $R$ is ideal-preserving, it will suffice to show that when restricting to $R$, the prime ideals dividing $I=J_1J_2$ remain distinct from one another and from their squares.
		
		Let $P$ be a prime $T$-ideal dividing $I$. Then $P$ must divide exactly one of $J_1$ or $J_2$; without loss of generality, suppose $P|J_1$. Then since $R_1$ is ideal-preserving, select $\alpha\in R_1\cap P\backslash P^2$ and $\beta\in R_1\cap J_2\backslash P$. Then $\alpha\beta\in R_1\cap J_2\subseteq R_1\cap R_2=R$, and $\alpha\beta\in P\backslash P^2$. Then $R\cap P\nsubseteq P^2$.
		
		Now suppose that we have distinct prime $T$-ideals $P_1\neq P_2$, both of which divide $I$. Again, each of these primes must divide either $J_1$ or $J_2$ but not both. Without loss of generality, assume $P_1|J_1$. If $P_2|J_1$, then we can proceed similarly to the previous case. Since $R_1$ is ideal-preserving, we can select $\alpha\in R_1\cap P_1\backslash P_2$ and $\beta\in R_1\cap J_2\backslash P_2$. Then $\alpha\beta\in R_1\cap J_2\subseteq R$ and $\alpha\beta\in P_1\backslash P_2$, so $R\cap P_1\nsubseteq P_2$. If $P_2|J_2$, we will instead make use of the fact that $R_2$ is ideal-preserving. In this case, select $\alpha\in R_2\cap P_1\backslash P_2$ and $\beta\in R_2\cap J_1\backslash P_2$. Then $\alpha\beta\in R_2\cap J_1\subseteq R$ and $\alpha\beta\in P_1\backslash P_2$, so $R\cap P_1\nsubseteq P_2$. Thus, $R$ is an ideal-preserving subring of $T$.
	\end{proof}
	
	The same pattern of inheritance is not observed with associated or locally associated subrings, as the following example shows.
	
	\begin{example}
		Let $R_1=\bZ[3\sqrt{2}]$, $R_2=\bZ[11\sqrt{2}]$, and $T=\bZ[\sqrt{2}]$. It can be verified that $R_1$ and $R_2$ are both associated subrings (and therefore locally associated subrings) of $T$ using the characterizations found in Theorem \ref{as equiv conditions}. However, $R=R_1\cap R_2=\bZ[33\sqrt{2}]$ is not a locally associated subring (and therefore not an associated subring) of $T$; this can be verified using the characterizations found in Theorem \ref{la equiv conditions}. These can also be checked using the table found at \cite{quadla}, which we will discuss more in-depth later.
	\end{example}
	
	Before we narrow our focus to look solely at orders in a number field, we have one more result to show that will hold more generally. This is arguably one of the most important results in this paper and serves to illustrate just how closely associated subrings are connected to ideal-preserving and locally associated subrings. First, a pair of lemmas.
	
	\begin{lemma}
		\label{Lemma 1}
		Let $T$ be a commutative ring with identity, $R\subseteq T$ a locally associated subring, and $I=(R:T)$. If for some $\alpha \in T$ there exists $\beta\in T$ relatively prime to $I$ such that $\alpha\beta\in R$, then there exists $u\in U(T)$ such that $\alpha u\in R$.
	\end{lemma}
	
	\begin{proof}
		Let $\alpha\in T$, and suppose that there is some $\beta\in T$ such that $\beta T+I=T$ and $\alpha\beta\in R$. Then by Condition 2 in Theorem \ref{la equiv conditions}, there exists $r\in R$ relatively prime to $I$ and $\gamma\in I$ such that $\beta r+\gamma=u\in U(T)$. Then $\alpha u=\alpha(\beta r+\gamma)=(\alpha\beta)r+\alpha\gamma$. Since $\alpha\beta\in R$, then $(\alpha\beta)r\in R$; since $\gamma\in I$, then $\alpha\gamma\in R$. Thus, $\alpha u\in R$.
	\end{proof}
	
	\begin{lemma}
		\label{Lemma 2}
		Let $T$ be a Dedekind domain, $R\subseteq T$ a subring with identity, and $I=(R:T)\neq \{0\}$. Also assume that $R$ is both an ideal-preserving and a locally associated subring of $T$. Then for any $T$-ideal $J$ dividing $I$, $R+J$ is also a locally associated and ideal-preserving subring of $T$.
	\end{lemma}
	
	\begin{proof}
		We have already seen that since $R+J$ must be an ideal-preserving subring of $T$, since $R\subseteq R+J\subseteq T$. Moreover, $(R+J:T)=J$. We need only show that $R+J$ is also a locally associated subring of $T$. First, note that if $I=T$, then $R=T$ and the result holds trivially (since the only $T$-ideal dividing $I$ is $T$ itself). Otherwise, we can factor $I=P_1^{a_1}\dots P_k^{a_k}$, with each $P_i$ being a prime $T$-ideal.
		
		First, assume that $J$ is relatively prime to $IJ^{-1}$; that is, there exists $1\leq j\leq k$ such that (after potentially rearranging the prime factors of $I$) $J=P_1^{a_1}\dots P_j^{a_j}$ and $IJ^{-1}=P_{j+1}^{a_{j+1}}\dots P_k^{a_k}$. We will show that $R+J$ is a locally associated subring of $T$ using Condition 3 from Theorem \ref{la equiv conditions}.
		
		Let $\alpha\in T$ be relatively prime to $J$. By the Chinese Remainder Theorem, we can pick $\beta\in T$ such that $\beta\equiv \alpha\modulo{J}$ and $\beta\equiv 1\modulo{IJ^{-1}}$. Notably, $\beta$ is relatively prime to $I$ and thus there exists $r\in R$ relatively prime to $I$ and $u\in U(T)$ such that $\beta=ru$. Now letting $\gamma\in J$ such that $\alpha=\beta+\gamma$, we have $\alpha=\beta+\gamma=(r+u^{-1}\gamma)u$. Since $r$ is relatively prime to $I$ as an element of $R$, then $r+u^{-1}\gamma$ is relatively prime to $J$ as an element of $R+J$. Then using Condition 3 from Theorem \ref{la equiv conditions} and the fact that $J=(R+J:T)$, we have that $R$ is a locally associated subring of $T$.
		
		We will now relax the condition that $J$ is relatively prime to $IJ^{-1}$. For some $1\leq j\leq k$, we will write (after potentially rearranging the prime factors of $I$) $J=P_1^{b_1}\dots P_j^{b_j}$ with $1\leq b_i\leq a_i$ for $1\leq i\leq j$. Then define $A:=P_1^{a_1}\dots P_j^{a_j}$ and note that $A$ is relatively prime to $IA^{-1}$. By the above case, this means that $R+A$ is a locally associated subring of $T$. Then for every $t\in T$ relatively prime to $A$, there exists $r\in R+A$ relatively prime to $A$ and $u\in U(T)$ such that $t=ru$. Since the set of prime $T$-ideals dividing $A$ is exactly the same set of prime $T$-ideals dividing $J$, an element $t\in T$ is relatively prime to $J$ if and only if it is relatively prime to $A$. Then any $t\in T$ relatively prime to $J$ has $r\in R+A\subseteq R+J$ relatively prime to $A$ (and thus relatively prime to $J$ as an element of $R+J$) and $u\in U(T)$ such that $t=ru$. Then $R+J$ is a locally associated subring of $T$.
	\end{proof}
	
	\begin{theorem}
		\label{ass iff ip and la}
		Let $T$ be a Dedekind domain, $R\subseteq T$ a subring with identity, and $I=(R:T)\neq\{0\}$. Also assume that $T$ is integral over $R$. Then $R$ is an associated subring of $T$ if and only if $R$ is both an ideal-preserving and a locally associated subring of $T$.
	\end{theorem}
	
	\begin{proof}
		First, note that we have already shown in Theorems \ref{ass implies ip} and \ref{ass implies la} that any associated subring is also ideal-preserving and locally associated. We need to show the converse.
		
		Assume that $R$ is both an ideal-preserving and a locally associated subring of $T$.  To start, we will also assume that $I$ is primary, i.e. $I=P^a$ for some prime $T$-ideal $P$ and some $a\in\bN$. Let $\alpha\in T$. We will consider three cases. First, if $\alpha$ is relatively prime to $I$, then since $R$ is locally associated, Condition 3 from Theorem \ref{la equiv conditions} tells us that there exists $u\in U(T)$ such that $\alpha u\in R$. If $\alpha\in I$, then $\alpha\cdot 1\in R$. This covers the first two cases.
		
		The final case to consider for primary $I$ is when $\alpha\in P\backslash I$. In this case, there must exist some $1\leq m<a$ such that $\alpha\in P^m\backslash P^{m+1}$. Since $R$ is ideal-preserving, $(\alpha)\nsubseteq P^{m+1}$ tells us that $R\cap (\alpha)\nsubseteq P^{m+1}$, i.e. there exists $\beta\in T$ such that $\alpha\beta\in R\backslash P^{m+1}$. Necessarily, $\beta\notin P$ and is thus relatively prime to $I$. By Lemma \ref{Lemma 1}, there exists $u\in U(T)$ such that $\alpha u\in R$. Then since such a unit exists for any $\alpha\in T$, $R$ is an associated subring of $T$.
		
		Now dropping the assumption that $I$ is primary, let $I=P_1^{a_1}\dots P_k^{a_k}$ be the factorization of $I$ into prime $T$-ideals. For each $1\leq i\leq k$, define $R_i=R+P_i^{a_i}$. Since $R$ is both ideal-preserving and locally associated, so is each $R_i$ by Lemma \ref{Lemma 2}. Moreover, since $R$ is ideal-preserving, Theorem \ref{ip intermediate cond} tells us that each $R_i$ has primary conductor ideal $P_i^{a_i}$. Thus, $R_i$ is an associated subring of $T$ for each $1\leq i\leq k$. Then for each $1\leq i\leq k$, there exists $u\in U(T)$ such that $r_i:=\alpha u_i\in R_i$. By the Chinese Remainder Theorem, we can now choose some $\beta\in T$ such that $\beta\equiv u_i\modulo{P_i^{a_i}}$ for every $1\leq i\leq k$. Since $\beta$ is congruent to a unit modulo each prime dividing $I$, $\beta$ must be relatively prime to $I$. Furthermore, $\alpha\beta\equiv \alpha u_i\equiv r_i\modulo{P_i^{a_i}}$, so by applying Theorem \ref{ip intersect cond ideals} several times, we get $$\alpha\beta\in \bigcap_{i=1}^{k}R_i=R.$$ Then using Lemma \ref{Lemma 1} again, there exists $u\in U(T)$ such that $\alpha u\in R$. Then $R$ is an associated subring of $T$.
	\end{proof}
	
	Then in the case that $T$ is a Dedekind domain, $R$ is a subring with identity whose conductor ideal $(R:T)$ is nonzero, and $R\subseteq T$ is an integral extension, $R$ is an associated subring of $T$ if and only if it is both ideal-preserving and locally associated. Whether this result will hold more generally is still under investigation; at this juncture, it seems unlikely to hold in full generality. However, this does give us the a particularly interesting case worth mentioning.
	
	\begin{corollary}
		\label{ao iff ipo and lao}
		Let $R$ be an order in a number field. Then $R$ is an associated order if and only if $R$ is both an ideal-preserving order and a locally associated order.
	\end{corollary}
	
	With this in mind, we now turn our attention more specifically to orders in a number field.
	
	\section{Orders in a Number Field}
	As we have alluded to throughout this paper, the concepts of associated, ideal-preserving, and locally associated subrings are especially interesting in the specific case when $R$ is an order in a number field $K$ and $T=\Rbar=\cO_K$. We saw this in the introduction with Theorems \ref{halter}, \ref{rago hfd}, and \ref{elasticities equal} in which $R$ being an associated order was used. In this section, we will explore how these properties can be applied in this specific case. It is important to note that since we are assuming that $R$ is an order in a number field and $T=\Rbar$, $T$ will be a Dedekind domain, $R$ will be a commutative subring of $T$ with identity, $I=(R:T)$ will be a nonzero ideal, and $T$ will be integral over $R$. Then we will be able to use any of equivalent characterizations of these properties from Theorems \ref{as equiv conditions}, \ref{ip equiv conditions}, and \ref{la equiv conditions} as well as the majority of the results developed thus far in this paper.
	
	To start, we should address one aspect of locally associated subrings that makes these objects less well-behaved than associated or ideal-preserving subrings: the fact that their defining property is not always inherited by intermediate subrings. That is, if $R\subseteq S\subseteq T$ is a tower of subrings, it is possible for $R$ to a locally associated subring of $T$ even if $S$ is not a locally associated subring of $T$. However, this inheritance will hold for locally associated orders in a number field; to show this, we need the following lemmas, both from \cite{conrad}.
	
	\begin{lemma}
		\label{rel prime means invertible}
		Let $R$ be an order in a number field $K$ with conductor ideal $I$. Then any $R$-ideal which is relatively prime to $I$ is invertible. That is, if $J$ is an ideal in $R$ such that $J+I=R$, then $JJ^{-1}=R$, with $J^{-1}=\{\alpha\in K|\alpha J\subseteq R\}$.
	\end{lemma}
	
	\begin{lemma}
		\label{int ideal rel prime}
		Let $R$ be an order in a number field $K$ and $J$ an ideal in $R$. Then every ideal class in $\Cl(R)$ contains a representative which is an integral ideal of $R$ that is relatively prime to $J$. That is, for any invertible $A\in \Inv(R)$, there exists $\alpha\in K$ such that $\alpha A\subseteq R$ and $\alpha A+J=R$.
	\end{lemma}
	
	\begin{theorem}
		\label{surjective class map}
		Let $R$ be an order in a number field $K$ with conductor ideal $I$, and let $S$ be an intermediate order, $R\subseteq S\subseteq\Rbar$. Then the mapping $\phi:\Cl(R)\to\Cl(S)$ such that $\phi([J])=[JS]$ is a surjective homomorphism (here, $[J]$ refers to the ideal class containing the fractional ideal $J$).
	\end{theorem}
	
	\begin{proof}
		First, note that $\phi$ is well-defined. If $[J_1]=[J_2]$ for two invertible fractional $R$-ideals $J_1$ and $J_2$, then $J_1=\alpha J_2$ for some $\alpha\in K$. Thus, $J_1S=\alpha J_2S$, so $\phi([J_1])=[J_1S]=[J_2S]=\phi([J_2])$. Furthermore, $\phi$ is a group homomorphism, since for any invertible fractional $R$-ideals $J_1$ and $J_2$, $\phi([J_1J_2])=[J_1J_2S]=[J_1S][J_2S]=\phi([J_1])\phi([J_2])$. All that remains to show is that $\phi$ is surjective.
		
		Let $[A]\in\Cl(S)$. Since $I$ is an ideal of $S$, we can use the second lemma above to assume without loss of generality that $A$ is an integral ideal of $S$ which is relatively prime to $I$. Then let $J=R\cap A$. Since $A$ is relatively prime to $I$, we know that there is some $\alpha\in A$ and $\beta\in I$ such that $\alpha+\beta=1$. Then $\alpha=1-\beta\in R$, so in fact $\alpha\in J$. Then $J$ is relatively prime to $I$ as an $R$-ideal. The first lemma above tells us that $J\in \Inv(R)$. Moreover, note the following:
		$$A=AR=A(J+I)=AJ+AI\subseteq JS\subseteq A.$$
		Then $A=JS$, so $\phi([J])=[JS]=[A]$. Thus, $\phi$ is a surjective homomorphism.
	\end{proof}
	
	\begin{corollary}
		\label{la in towers}
		Let $R$ be a locally associated order in a number field $K$. Then if $S$ is an intermediate order to $R$, i.e. $R\subseteq S\subseteq \Rbar$, $S$ is also locally associated.
	\end{corollary}
	
	\begin{proof}
		By Theorem \ref{la equiv conditions}, we know that since $R$ is locally associated, $\abs{\Cl(R)}=\abs{\Cl(\Rbar)}$. By the theorem, $\abs{\Cl(R)}\geq \abs{\Cl(S)}$; by Lemma \ref{exact sequence}, $\abs{\Cl(S)}\geq \abs{\Cl(\Rbar)}$. Then $\abs{\Cl(S)}\geq \abs{\Cl(\Rbar)}=\abs{\Cl(R)}\geq\abs{\Cl(S)}$, so $\abs{\Cl(S)}=\abs{\Cl(\Rbar)}$. Then again using Theorem \ref{la equiv conditions}, we have that $S$ is locally associated.
	\end{proof}
	
	Thus, in the realm of orders in a number field, the locally associated property is inherited by intermediate orders just like the associated and ideal-preserving properties. Among other applications, this will allow us to more easily search for orders with are locally associated. For instance, once we establish that $\bZ[5\sqrt{2}]$ is not a locally associated order (as we will momentarily), there is no need to check whether $\bZ[10\sqrt{2}]$ is a locally associated order, since $\bZ[10\sqrt{2}]\subseteq \bZ[5\sqrt{2}]\subseteq \bZ[\sqrt{2}]$. We can immediately draw the conclusion that $\bZ[10\sqrt{2}]$ is not locally associated.
	
	The ideal-preserving property also has additional features and implications in the case of orders in a number field. Perhaps unsurprisingly, these will be related to the conductor ideal of the order in question, as the following results show.
	
	\begin{lemma}
		Let $R$ be an order in a number field with conductor ideal $I$. Then $\abs{\quot{\Rbar}{I}}$ is not a rational prime.
	\end{lemma}
	
	\begin{proof}
		Assume that there exists a number field $K$ with order $R$ such that the conductor ideal $I$ of $R$ has prime norm in $\Rbar$; that is, $\abs{\quot{\Rbar}{I}}=p$ for some rational prime $p$. One will note that if $R=\Rbar$, then $I=\Rbar$, i.e. $I$ would have norm 1. Thus, $R$ is a non-maximal order. Now considering the additive groups of $\Rbar$, $R$, and $I$, we have that $p=\abs{\quot{\Rbar}{I}}=\abs{\quot{\Rbar}{R}}\cdot\abs{\quot{R}{I}}$. Since both factors on the right-hand side of this equality must be positive integers which multiply to be the prime integer $p$, it must be the case that one of these factors is equal to 1. However, we have already established that $\Rbar\neq R$, so we must have $R=I$. However, note that $1\in R\backslash I$, a contradiction. Then $\abs{\quot{\Rbar}{I}}$ cannot be prime.
	\end{proof}
	
	\begin{theorem}
		Let $R$ be an ideal-preserving order in a number field $K$ with conductor ideal $I$, and let $P$ be a prime divisor of $I$ which lies over the rational prime $p$. Then the inertial degree $f(P|p)>1$; that is, $\abs{\quot{\Rbar}{P}}=p^f$ for some $f>1$.
	\end{theorem}
	
	\begin{proof}
		Suppose that $I$ has a prime divisor $P$ which lies over $p\in \bZ$ and $f(P|p)=1$. This means that $\abs{\quot{\Rbar}{P}}=p$. By Theorem \ref{ip intermediate cond}, we know that $R+P$ is an order in $K$ with conductor ideal $P$. However, the above lemma tells us that this is impossible, since $P$ has prime norm. Then any prime divisor $P$ of $I$ must have $f(P|p)>1$. 
	\end{proof}
	
	This result narrows down the possible conductor ideals of ideal-preserving orders, making such orders easier to find. For instance, it immediately tells us that $\bZ[2\sqrt{2}]$ is not an ideal-preserving order, since its conductor ideal is $I=(2)=(\sqrt{2})^2$, with $f((\sqrt{2})|2)=1$.
	
	\begin{theorem}
		Let $R$ be an order in a number field with conductor ideal $I$. If every prime divisor of $I$ is principal in $\Rbar$ and can be generated (as an $\Rbar$-ideal) by an element of $R$, then $R$ is ideal-preserving.
	\end{theorem}
	
	\begin{proof}
		Assume that $I$ is as described in the theorem statement. Recall that to show that $R$ is ideal-preserving, we simply need to show that for any primes $P_1$ and $P_2$ dividing $I$, $R\cap P_1\nsubseteq P_2$ and $R\cap P_1\nsubseteq P_1^2$. By the assumption, such $P_1$ and $P_2$ are principal and generated by an element of $R$. Let $\pi\in R$ such that $P_1=\pi\Rbar$; then $\pi\in R\cap P_1$. However, if $\pi\in P_2$ or $\pi\in P_1^2$, that would imply that $P_1$ itself is contained in one of these ideals, a contradiction. Then $R\cap P_1\nsubseteq P_2$ and $R\cap P_1\nsubseteq P_1^2$, meaning that $R$ is ideal-preserving.
	\end{proof}
	
	This theorem has a strong assumption associated with it; that being said, it will be extremely useful to us in the case that $R$ is an order in a quadratic number field $K$. In fact, it will provide us with an easy-to-check equivalent characterization of ideal-preserving orders.
	
	\begin{theorem}
		\label{ip iff inert primes}
		Let $K=\bQ[\sqrt{d}]$ for some squarefree $d\in\bZ$ and $R$ be the index $n$ order in $K$. Then $R$ is ideal-preserving if and only if every rational prime divisor of $n$ is inert in $K$. 
	\end{theorem}
	
	\begin{proof}
		First, assume that $R$ is ideal-preserving and let $p$ be a rational prime divisor of $n$. Since $K$ is a quadratic number field, $p$ must either be inert or be a product of two (not necessarily distinct) prime $\Rbar$-ideals of norm $p$. Since $R$ is ideal-preserving, then its conductor ideal $n\Rbar$ cannot have any prime divisors of prime norm. Thus, every rational prime dividing $n$ must be inert.
		
		For the reverse implication, assume that $R$ has index $n$ such that every rational prime divisor of $n$ is inert in $K$. Then every prime $\Rbar$-ideal dividing the conductor ideal $I=(n)$ is principal and of the form $P=(p)$, with $p\in\bZ\subseteq R$. Then the previous theorem tells us that $R$ is ideal-preserving.
	\end{proof}
	
	This result makes checking whether an order in a quadratic number field is ideal-preserving almost trivial, since we can check whether a rational prime $p$ is inert in the quadratic number field $\bQ[\sqrt{d}]$ by using the Legendre symbol (unless $p=2$, in which case we need only check whether $d\equiv 5\modulo{8}$). In the interest of completeness, we can now use this characterization to show in more detail some examples that were glossed over earlier. These examples will also give us an idea of the process we might use to find ideal-preserving or locally associated orders, which we will explore more in the next section.
	
	\begin{example}
		\label{ip not la}
		Let $R=\bZ[5\sqrt{2}]$, the index 5 order in $K=\bQ[\sqrt{2}]$. Then $R$ is an ideal-preserving order which is not locally associated (and therefore not associated). Using Theorem \ref{ip iff inert primes}, we note that $\kron{2}{5}=-1$ and thus $5$ is an inert prime in $K$. Then $R$ is ideal-preserving. The fundamental unit in $\Rbar$ is $u=1+\sqrt{2}$, and the smallest power of $u$ lying in $R$ is $u^3=7+5\sqrt{2}$. Then $R$ is not locally associated, since
		$$\abs{\quot{U(\Rbar)}{U(R)}}=3\neq 6=\frac{24}{4}=\frac{\abs{U(\quot{\Rbar}{I})}}{\abs{U(\quot{R}{I})}}.$$
	\end{example}
	
	\begin{example}
		\label{la not ip}
		Let $R=\bZ[2\sqrt{2}]$, the index 2 order in $K=\bQ[\sqrt{2}]$. Then $R$ is a locally associated order which is not ideal-preserving (and therefore not associated). Note that $(2)=(\sqrt{2})^2$, so 2 is not an inert prime in $K$. Then using Theorem \ref{ip iff inert primes}, $R$ is not ideal-preserving. The fundamental unit in $\Rbar$ is $u=1+\sqrt{2}$, and the smallest power of $u$ lying in $R$ is $u^2=3+2\sqrt{2}$. Then $R$ is locally associated, since
		$$\abs{\quot{U(\Rbar)}{U(R)}}=2=\frac{2}{1}=\frac{\abs{U(\quot{\Rbar}{I})}}{\abs{U(\quot{R}{I})}}.$$
	\end{example}
	
	\section{Finding Associated, Ideal-Preserving, and Locally Associated Orders}
	As we have seen throughout this paper, associated, ideal-preserving, and locally associated orders are objects of interest with notable properties. Specifically, as we saw in Theorems \ref{halter}, \ref{rago hfd}, and \ref{elasticities equal}, these properties can give useful information about the elasticity of an order. It might be useful, then, to be able to find large numbers of these orders quickly or, when possible, find easy-to-use characterizations of where these orders can be found.
	
	For the purposes of this paper, we will restrict our search to the simplest orders: those lying in quadratic number fields. These orders are relatively easy to work with and thus have been well-explored historically. The study of factorization in these orders is of ongoing interest, including in \cite{choi2024class} and \cite{pollack}; the work in this section continues in this tradition. Orders in number fields of higher degree are generally more difficult to work with and represent one direction in which this research will naturally move in the future. For a brief handling of some special cases and discussion of how one might proceed in this more complex case, see \cite{dissertation}.
	
	As we saw in Theorem \ref{ip iff inert primes}, ideal-preserving orders are easy to characterize in the quadratic case. Since associated orders are simply orders which are both ideal-preserving and locally associated by Corollary \ref{ao iff ipo and lao}, this means we really only need to be concerned with identifying locally associated orders in the quadratic case. In this interest, we define the following function.
	
	\begin{definition}
		\label{l function}
		We define the function $L:\bN\times \bZ\to \bN$ as follows:
		\begin{enumerate}
			\item $L(1,d)=1$ for every $d\in\bZ$.
			\item For any $a\in\bN$:$$L(2^a,d)=\begin{cases} 2^{a-1},&d\equiv 1\modulo{8};\\3\cdot 2^{a-1},&d\equiv 5\modulo{8};\\2^a,&\ow.\end{cases}$$
			\item For any $a\in\bN$ and odd prime $p$, $L(p^a,d)=p^{a-1}\left(p-\kron{d}{p}\right)$.
			\item If $m,n\in\bN$ are coprime, then $L(mn,d)=L(m,d)\cdot L(n,d)$.
		\end{enumerate}
	\end{definition}
	
	In other words, if we fix $d\in\bZ$, the function $L(\cdot,d)$ is a multiplicative arithmetic function. The following theorem demonstrates the usefulness of this function.
	
	\begin{theorem}
		Let $R$ be the index $n$ order in the quadratic number field $K=\bQ[\sqrt{d}]$, where $d\in\bZ$ is squarefree. Then
		$$\frac{\abs{U(\quot{\Rbar}{I})}}{\abs{U(\quot{R}{I})}}=L(n,d).$$
	\end{theorem}
	
	\begin{proof}
		First, if $n=1$, then $\Rbar=R=I$, so it is clear that
		$$\frac{\abs{U(\quot{\Rbar}{I})}}{\abs{U(\quot{R}{I})}}=L(1,d)=1.$$
		Now for any $n\geq 2$, note that $R=\bZ+n\Rbar=\bZ+I$. Then $\quot{R}{I}\cong \bZ_n$, the ring of integers modulo $n$; in particular, this means that $\abs{U(\quot{R}{I})}=\abs{U(\bZ_n)}=\phi(n)$, where $\phi$ is Euler's totient function. In all cases, this covers the denominator of the fraction that we are considering; we now need to focus on the numerator.
		
		Suppose that $p$ is a rational prime and $n=p^a$ for $a\in\bN$. Then the structure of $\quot{\Rbar}{I}=\quot{\Rbar}{(n)}$ will depend on how the rational prime $p$ decomposes into prime ideals in $K$. If $p$ is an inert prime, i.e. $P=(p)$ is a prime ideal: $$\abs{U(\quot{\Rbar}{I})}=\abs{\quot{\Rbar}{I}}-\abs{\quot{P}{I}}=p^{2a}-p^{{2(a-1)}}=p^{2(a-1)}(p^2-1).$$
		If $p$ is a split prime, i.e. $(p)=PQ$ for distinct prime ideals $P,Q$:
		$$\abs{U(\quot{\Rbar}{I})}=\abs{U(\quot{\Rbar}{P^a})}\cdot \abs{U(\quot{\Rbar}{Q^a})}=\left(\abs{\quot{\Rbar}{P^a}}-\abs{\quot{P}{P^a}}\right)^2=(p^a-p^{a-1})^2=p^{2(a-1)}(p-1)^2.$$
		Finally, if $p$ is ramified, i.e. $(p)=P^2$ for some prime ideal $P$:
		$$\abs{U(\quot{\Rbar}{I})}=\abs{\quot{\Rbar}{I}}-\abs{\quot{P}{I}}=p^{2a}-p^{{2a-1}}=p^{2a-1}(p-1).$$
		
		Then if $n=2^a$ for some $a\in\bN$, we can use the fact that $2$ is inert in $K$ if and only if $d\equiv 5\modulo{8}$; 2 is split in $K$ if and only if $d\equiv 1\modulo{8}$; and 2 is ramified in $K$ otherwise. Then using the values of $\abs{U(\quot{\Rbar}{I})}$ and $\abs{U(\quot{R}{I})}$ found above:
		$$\frac{\abs{U(\quot{\Rbar}{I})}}{\abs{U(\quot{R}{I})}}=L(2^a,d)=\begin{cases} 2^{a-1},&d\equiv 1\modulo{8};\\3\cdot 2^{a-1},&d\equiv 5\modulo{8};\\2^a,&\ow.\end{cases}.$$
		Similarly, note that for an odd prime $p$, $p$ is inert in $K$ if and only if $\kron{d}{p}=-1$; $p$ is split in $K$ if and only if $\kron{d}{p}=1$; and $p$ is split in $K$ if and only if $\kron{d}{p}=0$ (i.e. iff $p|d$). Then if $n=p^a$ for some odd prime $p$ and $a\in\bN$:
		$$\frac{\abs{U(\quot{\Rbar}{I})}}{\abs{U(\quot{R}{I})}}=L(p^a,d)=p^{a-1}\left(p-\kron{d}{p}\right).$$
		
		Finally, assume that $n=m_1m_2$ for relatively prime integers $m_1$, $m_2$, and assume that we have shown the result for the index $m_1$ and $m_2$ orders in $K$. Then since the ideals $(m_1)$ and $(m_2)$ are relatively prime in $\Rbar$, the Chinese Remainder Theorem tells us that $\abs{U(\quot{\Rbar}{I})}=\abs{U(\quot{\Rbar}{(m_1)})}\cdot \abs{U(\quot{\Rbar}{(m_2)})}$. Similarly, $\abs{U(\quot{R}{I})}=\phi(n)=\phi(m_1)\phi(m_2)$. Then:
		$$\frac{\abs{U(\quot{\Rbar}{I})}}{\abs{U(\quot{R}{I})}}=\frac{\abs{U(\quot{\Rbar}{(m_1)})}}{\phi(m_1)}\cdot \frac{\abs{U(\quot{\Rbar}{(m_2)})}}{\phi(m_2)}=L(m_1,d)\cdot L(m_2,d)=L(n,d).$$
		Then by inducting on the prime factorization of $n$, the result holds for every $n\in\bN$ and squarefree $d\in\bZ$ (and thus for every quadratic order).
	\end{proof}
	
	Now recall from Dirichlet's Unit Theorem that for a quadratic number field $K=\bQ[\sqrt{d}]$, the unit group in the ring of integers is either a finite cyclic group (if $d<0$; i.e. if $K$ is non-real) or $\{\pm u^k|k\in\bZ\}$ for some fundamental unit $u$ (if $d>0$; i.e. if $K$ is real). Then since $\pm 1\in R$, it is immediate that $\abs{\quot{U(\Rbar)}{U(R)}}$ is equal to the minimal power of the fundamental unit $u$ lying in $R$ (if $d<0$, we will refer to the generator of the finite cyclic unit group as the fundamental unit). This gives us the following corollary.
	
	\begin{corollary}
		\label{la iff l function}
		Let $R$ be the index $n$ order in the quadratic number field $K=\bQ[\sqrt{d}]$, where $d\in\bZ$ is squarefree, and let $u\in U(\Rbar)$ be the fundamental unit (if $d<0$, we will refer to the cyclic generator of $U(\Rbar)$ as the fundamental unit). Then if $m\in\bN$ is minimal such that $u^m\in R$, then $m|L(n,d)$. Moreover, $R$ is locally associated if and only if $m=L(n,d)$.
	\end{corollary}
	
	\begin{proof}
		Recall from Dirichlet's Unit Theorem that for a quadratic number field $K=\bQ[\sqrt{d}]$, the unit group in the ring of integers is either a finite cyclic group (if $d<0$; i.e. if $K$ is non-real) or $\{\pm u^k|k\in\bZ\}$ for some fundamental unit $u$ (if $d>0$; i.e. if $K$ is real). Then it is immediate that $\abs{\quot{U(\Rbar)}{U(R)}}=m$. By Condition 4 in Theorem \ref{la equiv conditions}, this tells us that $R$ is locally associated if and only if $m=L(n,d)$. The fact that $m|L(n,d)$ is a consequence of Lemma \ref{exact sequence}. In particular, in the proof of this lemma found in \cite{thesis}, it was shown that for any order $R$ in a number field $K$, $\abs{\quot{U(\Rbar)}{U(R)}}$ must divide $\frac{\abs{U(\sfrac{\Rbar}{I})}}{\abs{U(\sfrac{R}{I})}}.$
	\end{proof}
	
	This gives us a simple way to check whether an order $R$ in a quadratic number field $K$ is locally associated. First, we calculate $L(n,d)$. Then, we identify the fundamental unit $u\in U(\Rbar)$. Finally, we find the minimal power of $u$ lying in $R$ and compare this minimal power with $L(n,d)$.
	
	Now let us focus on the case when $d<0$; i.e. when $K$ is a non-real quadratic order. By Dirichlet's Unit Theorem, any such field has finite unit group $U(\cO_K)$ consisting only of roots of 1. In particular, when $d=-1$, $U(\bZ[i])=\{\pm 1,\pm i\}$; when $d=-3$ and $\omega=\frac{-1+\sqrt{-3}}{2}$, $U(\bZ[\omega])=\{\pm 1, \pm \omega, \pm \omega^2\}$; and when $d\notin\{-1,-3\}$, $U(\cO_K)=\{\pm 1\}$. Furthermore, if $R$ is the index $n$ order in $K$ with $n>1$, $U(R)=\{\pm 1\}$ (since only the maximal orders in $\bQ[i]$ and $\bQ[\omega]$ contain the fundamental unit). Then in any order $R$ in a non-real quadratic number field, we immediately know the value of $\abs{\quot{U(\Rbar)}{U(R)}}$ and can use this to determine when $R$ is locally associated.
	
	\begin{theorem}
		\label{la non-real quad}
		Let $R$ be the order of index $n>1$ in the quadratic number field $K=\bQ[\sqrt{d}]$, where $d<0$ is squarefree. Then $R$ is locally associated if and only if one of the following conditions holds:
		\begin{enumerate}
			\item $d\equiv 1\modulo{8}$ and $n=2$;
			\item $d=-1$ and $n=2$;
			\item $d=-3$ and $n\in\{2,3\}$.
		\end{enumerate}
	\end{theorem}
	
	\begin{proof}
		First, note that if $d\notin\{-1,-3\}$, then the fundamental unit in $\Rbar$ is $-1$. Then the minimal power $m$ such that $(-1)^m\in R$ is $m=1$ (since $-1\in\bZ\subseteq R$). Then $R$ is locally associated if and only if $L(n,d)=1$; by inspection`, this is true if and only if $n=1$ or $n=2$ and $d\equiv 1\modulo{8}$.
		
		Now assume that $d=-1$; then the fundamental unit in $\Rbar$ is $i$. Since $i$ is only in the maximal order $R=\Rbar$, then for $n>2$, the minimal power $m$ of $i$ lying in $R$ is $m=2$. By inspection, $L(n,d)=2$ is only possible when $n=2$ and $d\not\equiv 1\modulo{4}$; when $n=4$ and $d\equiv 1\modulo{8}$; when $n=3$ and $\kron{d}{3}=1$; and when $n=6$ and both $d\equiv 1\modulo{8}$ and $\kron{d}{3}=1$. Since $d=-1\equiv 7\modulo{8}$ and $\kron{-1}{3}=-1$, then the only order in $K$ which is locally associated is that of order $n=2$.
		
		Finally, assume that $d=-3$; then the fundamental unit in $\Rbar$ is $u=\frac{1+\sqrt{-3}}{2}$. Since $u$ is only in the maximal order $R=\Rbar$, then for any $n>2$, the minimal power $m$ of $u$ lying in $R$ is $m=3$. By inspection, $L(n,d)=3$ is only possible when $n=2$ and $d\equiv 5\modulo{8}$; when $p=3$ and $\kron{d}{3}=0$; and when $n=6$ and both $\kron{d}{3}=0$ and $d\equiv 1\modulo{8}$. Since $d=-3\equiv 5\modulo{8}$ and $\kron{-3}{3}=0$, then the orders in $K$ which are locally associated are exactly those of order $n=2$ and $n=3$.
	\end{proof}
	
	Keeping in mind Corollary \ref{ao iff ipo and lao} and Theorem \ref{ip iff inert primes}, this immediately yields the following.
	
	\begin{corollary}
		\label{associated non-real quad}
		Let $R$ be the order of index $n>1$ in the quadratic number field $K=\bQ[\sqrt{d}]$, where $d<0$ is squarefree. Then $R$ is associated if and only if $d=-3$ and $n=2$.
	\end{corollary}
	
	Recall that in this case, $\Rbar=\bZ[\frac{1+\sqrt{-3}}{2}]$ is an HFD. Using Theorem \ref{halter}, this gives an alternate proof of a result originally from \cite{coykendallhfd}.
	
	\begin{corollary}
		Let $R=\bZ[\sqrt{-3}]$, the index $2$ order in the number field $K=\bQ[\sqrt{-3}]$. Then $R$ is an HFD. Moreover, this is the only non-maximal half-factorial order in a non-real quadratic number field.
	\end{corollary}
	
	With the non-real quadratic case completely taken care of, we can now focus on the real quadratic case; that is, orders $R$ in a number field $K=\bQ[\sqrt{d}]$ where $d$ is a squarefree positive integer. In this case, recall that the unit group $U(\Rbar)$ has a single fundamental unit $u$ and the order of index $n$ in $K$ is locally associated if and only if $m=L(n,d)$, where $m$ is the minimal power of $u$ lying in $R$. This makes checking whether such an order is locally associated (and thus if it is associated) relatively straightforward; however, for large values of $n$ and $d$, it may be computationally infeasible to do so.
	
	In order to identify a large number of orders which are associated, ideal-preserving, and locally associated, MATLAB programs were written to systematically search through orders and identify their properties. Originally, this was done in \cite{dissertation}, producing a table of orders of index $n\leq 30$ in quadratic number fields $K=\bQ[\sqrt{d}]$ with $d<30$. This code has since been modified and improved upon, using SageMath when necessary to handle numbers larger than those MATLAB could work with. This improved code has produced a table identifying every order of index $n\leq 10000$ in quadratic number fields $K=\bQ[\sqrt{d}]$ with $d<1000$ as being associated, ideal-preserving, locally associated, or none of the above. Using Theorem \ref{halter}, it was then possible to determine which of these orders were specifically half-factorial. The resulting table identifying these orders, along with the code used to produce this table, can be found at \cite{quadla}. From this table, we get the following figure.
	
	\begin{proposition}
		There exist exactly $29$,$163$ half-factorial orders of index $1<n<=10000$ in quadratic number fields $K=\bQ[\sqrt{d}]$ with $1<d<1000$.
	\end{proposition}
	
	From this table, certain patterns of which orders are locally associated emerge. Research into what these patterns might be and how they might be leveraged into a simple characterization of locally associated orders in quadratic number fields is ongoing.
	
	\bibliographystyle{plain}
	\bibliography{bibliography}
\end{document}